\documentclass[12pt,reqno]{amsart}

\renewcommand{\div}{\mathop\mathrm{div}}

\def\R{\mathbb R}
\def\eb{\varepsilon}
\def\Bbb{\mathbb}
\def\({\left(}
\def\){\right)}
\def\Cal{\mathcal}
\def\divv{\operatorname{div}}
\def\Rot{\operatorname{curl}}

\def\Dx{\Delta_x}
\def\Nx{\nabla_x}
\def\Dt{\partial_t}

\usepackage{amssymb}
\usepackage{amsmath}
\usepackage{latexsym}
\usepackage{amsthm}
\usepackage{epsfig}
\usepackage{color}
\usepackage{comment}
\usepackage{amsfonts,amssymb,mathrsfs,amscd}

\usepackage{enumitem}
\usepackage{bm}

\newtheorem{proposition}{Proposition}[section]
\newtheorem{theorem}[proposition]{Theorem}
\newtheorem{corollary}[proposition]{Corollary}
\newtheorem{lemma}[proposition]{Lemma}
\theoremstyle{definition}
\newtheorem{definition}[proposition]{Definition}
\newtheorem{remark}[proposition]{Remark}

\numberwithin{equation}{section}
\def\R{\mathbb R}
\def\eb{\varepsilon}
\def\supp{\operatorname{supp}}
\def\Ree{\operatorname{Re}}
\def\({\left(}
\def\){\right)}
\def\Cal{\mathcal}
\def\divv{\operatorname{div}}
\def\Rot{\operatorname{curl}}

\def\Dx{\Delta_x}
\def\Nx{\nabla_x}
\def\Dt{\partial_t}

\begin{document}
%
 \title[Lower bounds on the attractor dimension]
 {Multi-vortices and lower bounds on the attractor dimensions  for  2D Navier--Stokes equations}
\author[ A.G. Kostianko, A.A. Ilyin, D. Stone,
 S.V. Zelik] {
Anna Kostianko${}^{1,5}$, Alexei Ilyin${}^{2,4,5}$, Dominic Stone${}^3$   and
 Sergey Zelik${}^{1,2,3,5}$}
\email{ilyin@keldysh.ru}
\email{a.kostianko@zjnu.edu.cn}
\email{d.stone@surrey.ac.uk}
\email{s.zelik@surrey.ac.uk}
\address{${}^1$ Zhejiang Normal University, Department of Mathematics, Zhejiang, China}
\address{${}^2$ Keldysh Institute of Applied Mathematics, Moscow, Russia}
\address{${}^3$ University of Surrey, Department of Mathematics, Guildford, UK}
\address{${}^4$ Institute for Information Transmission Problems, Moscow, Russia}
\address{${}^5$ HSE University, Nizhny Novgorod, Russia}
\begin{abstract}
We present a principally new method
   for obtaining the lower bounds
for the attractors dimensions of the equations related with
hydrodynamics, which is not based on the Kolmogorov flows, and apply
it to the classical 2D Navier--Stokes equations in a bounded domain
as well as for the  Navier--Stokes equations with Ekman damping in
the whole plane. In particular, in the case of bounded domains, we give the lower
bounds, which are similar to the well-known estimate on a torus.  In both cases  our estimates are sharp.  Note that no lower bounds for these two cases were known before.
\par
We suggest to use the so-called multi-vortex, which consists of a well-separated Vishik vortices (i.e., spectrally unstable localized in space flows constructed by M.M. Vishik), as the analogue of the Kolmogorov flows. Note also that this method reproduces  the known result on the torus and that it is applicable to many other equations of hydrodynamics.
\end{abstract}
\maketitle
\tableofcontents
 \setcounter{equation}{0}
\section{Introduction}\label{sec1}
It is well-known that dynamical systems (DS) arising in many areas of natural sciences generate strong instabilities and demonstrate very irregular and complicated (chaotic) behavior. The classical examples are turbulent flows arising in hydrodynamics, in particular,  generated by  Navier-Stokes equations.
One of the most surprising  lessons of the 20th century is that
similar behavior with strong random features can be generated by
even relatively simple and deterministic ordinary differential
equations (ODEs). This phenomenon, which is referred nowadays as ``deterministic chaos", is intensively studied starting from the second part of the 20th century and a lot of prominent results are obtained in this direction as well as a lot of methods for investigating it are developed (such as hyperbolic theory, the Lyapunov exponents, homoclinic bifurcation theory, strange attractors, etc., see \cite{ER85,GH83,KH95,Rue81,Shi01} and references therein). However, this phenomenon occurred to be much more complicated than it was thought from the very beginning, so,  despite many efforts made, the existing theory is mainly restricted to low dimensional model examples and there are still lack of  effective methods for studying the deterministic chaos in higher dimensional systems of ODEs.
\par
The situation is a priori more complicated when we  deal with dynamical systems generated by partial differential equations (PDEs). For such systems, the initial phase space is infinite-dimensional (e.g. $L^2(\Omega)$, where $\Omega$ is a domain in $\R^d$) which causes a lot of extra difficulties. In addition, together with a temporal variable $t$ and related temporal chaos, we now have spatial variables $x\in\Omega$, so the so-called spatial chaos may naturally appear. Also, we may have an interaction between spatially and temporal chaotic modes, which forms the so-called spatio-temporal chaos. As a result, a new type of purely infinite-dimensional dynamics with principally new level of complexity (which is unreachable  in systems of ODEs studied in classical dynamics) may appear, see \cite{MZ07,TZ10,Zel04} and references therein.
\par
Nevertheless, there exists a wide class of PDEs, namely, the class of
{\it dissipative} PDEs (usually in bounded domains), where
despite the infinite-dimensionality of the initial phase space,
the effective limit dynamics is in some sense finite-dimensional
and can be described by finitely many parameters (e.g., the  ``order parameters" suggested by I. Progogine for describing the evolution of dissipative structures, see \cite{Pri77}) whose evolution obeys a system of ODEs (which is called an inertial form (IF) of the initial dissipative system).
\par
\vskip 10pt
Thus, the natural strategy for the study of such systems would be
\begin{enumerate}
  \item To find an effective  scheme for doing the above mentioned finite-dimensional reduction;
   \item To write out the reduced ODEs and study them using the methods of the classical DS theory;
   \item To return back and get the conclusion about the initial dissipative dynamics.
 \end{enumerate}
    In particular, an application of this strategy to Navier-Stokes equations gives us a hope to understand the nature of turbulence. In addition, in this case, the Kolmogorov theory predicts the  finite-dimensionality of the limit dynamics and gives some hints about the number of degrees of freedom of the reduced system of ODEs depending on the physical parameters of the considered system, see \cite{Frish,Tem}. This is one more source of motivation to study the finite-dimensional reduction problem: obtaining  sharp upper and lower bounds for the number of degrees of freedom we may check these predictions rigorously, which would give us a way to  verify at least partially the validity of the Kolmogorov theory.
\par
 The mathematical theory of infinite-dimensional dissipative DS has been intensively developing during the last fifty years, see \cite{BV,FP67,Lad72,Tem,Z} and references therein. However, despite many efforts done by prominent scientists, the nature of the above mentioned finite-dimensional reduction remains a mystery. In particular, this reduction seems to require principally different tools depending on the smoothness of the reduced IF we want to obtain. Indeed, the most straightforward and natural interpretation of this reduction is related with inertial manifolds (IMs), which are  globally stable normally hyperbolic submanifolds of finite dimension in the phase space of the considered problem, see \cite{Man77, FST88,Zel14}. If such an object is found, the reduced IF is obtained by the projection of the initial equations to the base of IM and any trajectory of the initial DS approaches exponentially fast   the appropriate trajectory of the IF. However, the existence of such an object requires splitting of the phase space to slow and fast variables (usually formulated in terms of spectral gap conditions), which are very restrictive and are usually satisfied only in the case of one space dimension. For some special classes of PDEs this problem can be overcome, see \cite{MPS88,K17,KSZ22} and references therein and $C^{1+\eb}$ IF can be constructed (see \cite{KZ24} for more smooth IFs), but even in these cases the dimension of the IM is usually extremely high and is not compatible with  our expectations.
 \par
 An alternative way for constructing the finite-dimensional reduction is based on a concept of a global attractor and on the Man\'e projection theorem. By definition, the global attractor is a compact invariant set in the phase space which attracts the images of all bounded sets as time tends to infinity. Thus, on the one hand, it contains all non-trivial dynamics (up to the transitional one, which is assumed non-essential in this theory) and, on the other hand, it is usually essentially smaller than the initial phase space. Moreover, one of the key results of the attractors theory is that, under the mild assumptions on the system considered, the fractal dimension $\dim_f(\mathcal A)$ of the attractor $\mathcal A$ is finite. Combining this result with the Man\'e projection theorem, we get one-to-one projection of the attractor to a generic plane of dimension $N>2\dim_f(\mathcal A)$ with H\"older continuous inverse. Thus, we get a {\it H\"older continuous} IF, defined on a compact set of $\R^N$, which captures the dynamics on the attractor. This justifies the commonly accepted paradigm to treat the fractal dimension of the global attractor as a number of degrees of freedom of the reduced finite-dimensional dynamics, see \cite{Rob,Tem,Zel14} and references therein. Note that such H\"older continuous reduction is far from being perfect since non-smooth ODEs may demonstrate much more complicated behavior than the smooth ones, see \cite{KZ18,Zel14} for relatively simple  examples where the fractal dimension of the attractor is finite, but the original dissipative system demonstrate the features, which are impossible in classical dynamics, so the above paradigm should be corrected. However, up to the moment, the fractal dimension is probably the only way, which allows us to get the realistic bounds for the number of degrees of freedom of the limit dynamics (especially for DS arising in hydrodynamics). This explains a great permanent interest to obtaining as good as possible upper and lower bounds for various classes of equations arising in various areas of modern science, see \cite{BV,CF85,MP76,Tem,IKZ,IMT,IPZ,IKZ22,IKZ24,IZ25} and references therein.
\par
In this paper, we present a new method of obtaining lower bounds for dissipative PDEs and demonstrate it on the model example of the 2D Navier-Stokes equations with and without Ekman damping:
\begin{equation}\label{-1.NS}
\Dt u+\divv(u\otimes u)+\Nx p+\mu u=\nu\Dx u+f,\ \ \divv u=0,\ \ u\big|_{t=0}=u_0
\end{equation}
in the domain $\Omega\subset\R^2$,
where $u=(u_1,u_2)$ is an unknown velocity field, $p$ is pressure, $\nu>0$, $\mu\ge0$ are given kinematic viscosity and Ekman damping respectively and $f$ are given external forces. We consider two cases.
\par
{\it Case I.} $\mu=0$ and $\Omega$ is a bounded simply connected domain with  smooth boundary, endowed with the Dirichlet boundary conditions. This case corresponds to the classical 2D Navier-Stokes equations. The natural dimensionless quantity here is the Grashof number
$$
G:=\frac{|\Omega|\|f\|_{L^2}}{\nu^2}
$$
and the estimates for the dimension of the attractors are traditionally written in terms of this quantity.
On the other hand, all of the known upper bounds for the attractor's dimension in bounded domains and Dirichlet boundary conditions are based on the $H^{-1}$-norm  of the right-hand side $f$ and the $L^2$-norm appears just as a simplification through the Poincare inequality. However, as we will see later, this ``simplification" may loose the important information and may destroy the sharpness of the obtained estimates, so, for the case of Dirichlet boundary conditions it looks preferable to use the $H^{-1}$-analogue of the Grashof number:
$$
G_{-1}:=\frac{|\Omega|^{1/2}\|f\|_{\dot H^{-1}}}{\nu^2}
$$
and, probably, keep $G$ for the case of periodic boundary conditions. Recall that the $\dot H^{-1}$-norm $f$ is computed as follows: let $f=\Nx p+\divv F$ for some tensor field $F\in [L^2(\Omega)]^4$ and $p\in L^2(\Omega)$ (obviously, such a representation exists for any $f\in H^{-1}$, to find it, we may use $H^{-1}\to H^1$ regularity for the linear Stokes equation). Then,
$$
\|f\|_{\dot H^{-1}}:=\inf_{F}\|F\|_{L^2},
$$
where the  infimum  is taken over all $F\in L^2$ satisfying $f=\Nx p+\divv F$.
\par
{\it Case II.} $\nu,\mu>0$ and $\Omega=\R^2$. This is the Ekman damped
version of 2D Navier--Stokes equations and the analogue of the Grashoff number is the following quantity:
$$
G_1=\frac{\|\Rot f\|_{L^2}^2}{\mu^3\nu}.
$$
Obtaining upper and lower bounds for the attractor's dimension
for the Navier--Stokes equations has a rich and interesting
history, see \cite{Liu,Lieb84,CFT,CF85,Lad82,BV,Tem} and
references therein. We mention here only that the most
effective  technical tool nowadays for the upper bounds
in this case is the volume contraction method combined with
the Lieb--Thirring inequalities  and/or
collective $L^\infty$ inequalities of Brezis--Gallouet type, which gives the following results:
\begin{equation}\label{-1.up}
{\rm Case \ I:\ \ } \dim_f(\Cal A)\le c G_{-1},\ \ \ {\rm Case\ II:\ \ } \dim_f(\Cal A)\le c_1G_1.
\end{equation}
Also, in the case of periodic boundary conditions, the first
estimate can be essentially improved \cite{CFT,Tem,DoerGib}:
\begin{equation}\label{-1.up1}
\dim_f(\Cal A)\le CG^{2/3}\ln^{1/3}(1+G),
\end{equation}
where $c,c_1,C$ are some absolute constants, see \cite{IMT,IKZ} for the explicit expressions for them.
\par
The lower bounds for the attractor dimension for Navier--Stokes type equations are obtained for the case of periodic boundary conditions only. In this case we have
\begin{equation}\label{-1.lowt}
{\rm Case\ I:\ \ } \dim_f (\Cal A)\ge c_2 G^{2/3},\ \ \ \ {\rm Case\ II:\ \ }\dim_f(\Cal A)\ge c_3 G_1,
\end{equation}
see \cite{Liu,IKZ,IMT}.

As usual, in contrast to the upper bounds, the lower bounds hold not for all right-hand sides, but for the specially chosen ones.The crucial thing for such estimates is to find/construct an equilibrium (or periodic orbit) with as big as possible instability index. Then the unstable manifold of this equilibrium belongs to the attractor and the fractal dimension of the attractor is minorated  by the dimension of the manifold, i.e. by the instability index of the equilibrium constructed.
\par
To the best of our knowledge, in all previous works, the lower bounds for equations related with hydrodynamics are obtained using the {\it Kolmogorov flows} and are based on the classical work \cite{MS}, where the crucial method of estimating the instability index for such flows were invented.
That works perfectly for the case of torus or sphere and gives the estimates for the number of degrees of freedom comparable with the predictions of the Kolmogorov theory, see \cite{Frish,Tem}, but does not work in a bounded domain or the whole plane since the Kolmogorov flows cannot be reasonably restricted to a bounded domain or extended preserving the finiteness of the energy to the whole plane.
By this reason, no lower bounds which grow as $G$ (resp. $G_1$) tend to infinity were known for the Navier--Stokes equation in the Case I (resp. Case II) despite many efforts done during the last 40 years. We also note that the dependence of the finite-dimensional reduction on the boundary conditions (BC) is still poorly understood, so the difference between the physical Dirichlet BC  and the ``simplified" periodic ones may a priori be drastic. For instance, on the level of IMs, there are examples of equations (1D reaction-diffusion-advection systems) where the IM always exists for Dirichlet or Neumann BC, but may not exist for the periodic case, see \cite{KZ18}.
\par
The main aim of this paper is to cover the above mentioned gap and to get the  sharp  lower bounds for the attractor in two aforementioned cases. Namely, we prove the following result, which was announced in \cite{K24}.
\begin{theorem}\label{Th-1.main} The fractal dimension of the global attractor $\Cal A$ for equation \eqref{-1.NS} in the phase space $\Phi:=L^2_\sigma(\Omega)$ of square integrable divergence free vector fields possesses estimates \eqref{-1.lowt}.  Moreover, in the case of a bounded domain $\Omega$ with Dirichlet boundary conditions, we also have
\begin{equation}\label{-1.sharp}
\dim_f(\Cal A)\ge c_3G_{-1}
\end{equation}
  see  section \ref{s5} for more details.
\end{theorem}
Thus, in Case II ($\Omega=\R^2$), we have sharp (with respect to $G_1\to\infty$) upper and lower bounds and in Case I (classical Navier-Stokes equations, bounded domain, Dirichlet BC), we still have a gap
$$
c_2G^{2/3}\le\dim_f(\Cal  A)\le c G_{-1}\le \tilde c G
$$
if the standard Grashof number $G$ is used. However, the usage of $G_{-1}$ instead gives the sharp bounds
$$
c_3G_{-1}\le \dim_f(\Cal A)\le cG_{-1}.
$$
The same $H^{-1}$-sharp estimate holds with periodic boundary conditions as well. We emphasize once more that the upper bound for the dimension of the attractor in terms of the quantity $G_{-1}$ is well-known, see e.g. \cite{BV} (see also \cite{RobD} and references therein for the comparison of $G$ and $G_{-1}$ and the discussion of their physical meaning and optimality). We give some extra details in Remark \ref{Rem7.GG} below.
\par
The proof of these results is based on a synergy of two basic ideas. The first is the existence of a fully spatially localized time independent velocity profile, the linearization of the Navier--Stokes system in $\R^2$ around which possesses at least one exponentially unstable mode -- the so-called Vishik vortex, see section \ref{s1} for details. This vortex was constructed in \cite{V1,V2}, see \cite{Al,A2} for various simplifications and generalizations, in order to verify the non-uniqueness of weak solutions for singularly forced Navier--Stokes and Euler solutions. The present work gives another non-trivial application for this object.
\par
The second basic idea is to construct an equilibrium with big instability index as a multi-vortex which consists of a sum of spatially well-separated Vishik vortices
\begin{equation}\label{-1.m}
\bar u_{\Xi}(x)=\sum_{\xi_i\in\Xi}\bar u_0(x-\xi_i),
\end{equation}
where $\bar u_0$ is the initial Vishik vortex and $\xi_i\in\Xi\subset\Omega$ are well-separated  vortex centers:
$$
|\xi_i-\xi_j|\ge L\gg1,\ \  i\ne j\ \  \text{and }\ \  d(\xi_i,\partial\Omega)\ge L\gg1
$$
(due to the proper scaling we may assume that $\nu=1$ and $\Omega$ is big, see section \ref{s5} for the details ). Crucial for our method is the fact that the instability index $\operatorname{ind}^+(\bar u_\Xi)$ of the multi-vortex is not less than $\#\Xi$ if $L$ is big enough and this gives all of the desired estimates.
\par
To get this result we use the theory of weak interaction of localized structures and the corresponding center manifold reduction, see \cite{MiZ,BLZ08} and references therein. However, in contrast to the standard case, where the interaction between localized structures is exponentially small, in our case this is not so due to the presence of pressure. Moreover, the kernel of the Leray--Helmholtz projector decays only quadratically as $|x|\to\infty$ and therefore is not integrable. The same problem arises with the adjoint eigenfunctions of the linearized equation near the initial Vishik vortex. By this reason, the standard theory will work only with finitely many vortices and, which is more important, all of the estimates will depend explicitly on the number of them and this is not acceptable for our purposes.
\par
Although the interaction of vortices is also a classical topic in hydrodynamics, see \cite{Sa92,Cs24,CW} and references therein, the methods invented there, seem not sufficient for our purposes, so we suggest a new approach, which combines the standard center manifold reduction for weak interaction of localized structures with some methods, which are specific for the Navier-Stokes equations. Developing
this approach is the most difficult and the most technical part of the paper. We believe that the theory presented here is not restricted to the model cases I and II and is applicable to many other hydrodynamical equations. We also mention that our approach works for the case of periodic BC as well and recovers in an alternative  more transparent way the known estimate for the attractor's dimension.
\par
To conclude, we give a toy model example, which clarifies where the exponent $2/3$
in \eqref{-1.lowt} comes from. Namely, let $\omega_0$ be a bounded domain
in $\R^2$ and let equation \eqref{-1.NS} with $\nu=1$, $\mu=0$ and the
right hand side $f_0$ possess a non-trivial global attractor $\Cal A_0$
(for instance, let this equation possess an equilibrium which is exponentially unstable).
 We consider equation \eqref{-1.NS} in the ``domain"
$$
\Omega:=\cup_{i=1}^N\omega_j,\ \ \omega_j=\omega_0+\xi_j,
$$
where the centers $\xi_i\in\R^2$ are chosen in such a way that $\omega_i\cap\omega_j\ne0$ if $i\ne j$. We also take
$$
f(x)=\sum_{i=1}^Nf_0(x-\xi_i).
$$
Then, since the equation in $\Omega$ is actually a Cartesian product of $N$ equations in the domain $\omega_0$, for its global attractor we have
\begin{multline*}
\Cal A\sim\Cal A_0^N,\ \dim_f\Cal A= N\dim_f\Cal A_0\ge N,\
|\Omega|\sim N,\\ \|f\|_{L^2(\Omega)}\sim N^{1/2},\  G\sim N^{3/2},
\ \|f\|_{\dot H^{-1}(\Omega)}\sim N^{1/2},\ G_{-1}\sim N,
\end{multline*}
and therefore
$$
\dim_f(\Cal A)\sim G^{2/3}\sim G_{-1}.
$$
  Unfortunately, we are unable to lift the upper bound for the dimension from this toy case to more realistic domains, but the multi-vortex approach developed in the paper allows us to do this with the lower bound.
\par
The paper is organized as follows. Section \ref{s0} is devoted to some standard facts from the theory of weighted Lebesgue and Sobolev spaces which will be used throughout of the paper. In section \ref{s1}, we collect basic facts about Vishik vortices and the spectral properties of the linearized operator, which are crucial for what follows.
\par
In section \ref{s2}, we start to study the spectral properties of the Oseen operator in $\R^2$ related with the multi-vortex considered near the unstable eigenvalue of the initial Vishik vortex. We remove the kernel by adding the corresponding approximation of the spectral projector and prove the invertibility of the obtained auxiliary operator. In section \ref{s3}, we extend this result to the case of bounded domains. In section \ref{s4}, we return to the initial Oseen operator and find the invariant subspaces which correspond to neutral and hyperbolic subspaces of the initial Vishik vortex. In particular, the key result about the instability index mentioned above is proved there.
\par
Finally, the main results of the paper are stated and proved in section~\ref{s5}.

\section{Preliminaries I. Weighted spaces and weighted estimates}\label{s0}
In this section, we introduce the key weight
functions and related weighted spaces which will be used
throughout  the paper. Namely, we  will systematically use
the following weight function
\begin{equation}\label{0.theta}
\theta(x):=(1+|x|)^{-3},\ \ \theta_{x_0}(x):=\theta(x-x_0).
\end{equation}
These functions satisfy the key inequality
\begin{equation}\label{0.tineq}
\theta(x)\theta(y)\le 4\theta(x-y)(\theta(x)+\theta(y)).
\end{equation}
Indeed, due to the
triangle inequality,
$$
(1+|x-y|)^3\le (1+|x|+|y|)^3\le 4((1+|x|)^3+|y|^3)\le 4((1+|x|)^3+(1+|y|)^3),
$$
which gives \eqref{0.tineq}, see e.g. \cite{Z1}.
\begin{lemma}\label{Lem0.sum} Let $L>0$ be a sufficiently big number. Then
\begin{equation}\label{0.sum-uni}
\sum_{n\in\Bbb Z^2}\theta(x-Ln)\le C,
\end{equation}
where the constant $C$ is independent of $x\in\R^2$. Moreover,
\begin{equation}\label{0.sum-small}
\sum_{n\in\Bbb Z^2,\ n\ne0}\theta(Ln)\le CL^{-3},
\end{equation}
where the constant $C$ is independent of $L$.
\end{lemma}
\begin{proof} Indeed, let $m\in L\Bbb Z^2$ be the nearest to $x$ node of the grid (without loss of generality we may assume that $m=0$). Then, $|x|\le L/\sqrt2$, so $|x-Ln|\ge L|n|-L/\sqrt2$ and
$$
\sum_{n\in\Bbb Z^2}\theta(x-Ln)\le \theta(0)+L^{-3}\sum_{n\ne0}\frac1{(|n|-1/\sqrt2)^3}\le 1+CL^{-3}
$$
and the proof of the second estimate is analogous.
\end{proof}
As the next step, we  introduce some basic weighted spaces.
\begin{definition} The space $L^2_{\theta_{x_0}}$, where $\theta_{x_0}(x):=\theta(x-x_0)$ is defined as a subspace of $L^2_{loc}(\R^2)$ by the finiteness of the following norm:
$$
\|u\|_{L^2_{\theta_{x_0}}}^2:=\int_{x\in\R^2}|u(x)|^2\theta_{x_0}(x)\,dx.
$$
The space $L^2_{\theta_{x_0}}(\Omega)$ is defined analogously. In the sequel, we will also use the spaces of vector fields $[L^2_{\theta_{x_0}}]^2$ or tensor fields $[L^2_{\theta_{x_0}}]^4$, but for simplicity we will denote them also as $L^2_{\theta_{x_0}}$ everywhere, where it does not lead to misunderstandings.
\end{definition}
The next simple result is also  useful for what follows.
\begin{lemma}\label{Lem0.eq} The following inequalities hold:
\begin{equation}\label{0.eq}
l\|u\|_{L^2_{\theta_{x_0}}}^2\le
\int_{x\in\R^2}\|u\|_{L^2(B^1_{x})}^2\theta_{x_0}(x)\,dx\le
 L\|u\|_{L^2_{\theta_{x_0}}}^2,
\end{equation}
where the positive constants $l$ and $L$ are independent of $x_0$ and $B^1_x$
is a ball of radius $1$ centered at $x\in\R^2$. Moreover, the analogous estimates
 hold for the spaces in a domain $\Omega$ if the domain satisfies some minimal
 regularity assumptions. Also the radius $1$ of the balls in the middle norm can
 be replaced by any $R$, the corresponding norms are equivalent.
\end{lemma}
For the proof of this lemma, see e.g \cite{EfZ,MZ1}.
The above lemma allows us to introduce the weighted Sobolev spaces in a convenient way.
\begin{definition}\label{Def0.Sob} The weighted Sobolev space $H^{s}(\R^2)$, $s\in\R$, is defined as a subspace of distributions for which the following norm is finite:
\begin{equation}\label{0.ws}
\|u\|_{H^{s,}_{\theta_{x_0}}}^2:=\int_{x\in\R^2}\|u\|^2_{H^{s}(B^1_x)}\theta_{x_0}(x)\,dx,
\end{equation}
where $H^s$ are the standard non-weighted Sobolev spaces.
We also recall the definition of uniformly local Sobolev spaces $H^{s}_b(\R^2)$, which are the subspaces of the space of distributions defined by the following norm:
\begin{equation}\label{0.ul}
\|u\|_{H^{s}_b}:=\sup_{x_0\in\R^2}\|u\|_{H^{s}(B^1_{x_0})}.
\end{equation}
Next, we introduce the space $L^2A^\infty_{\theta_{x_0}}$
as a subspace of $L^2_{loc}$ where the following norm is finite:
\begin{equation}
\|u\|_{L^2A^\infty_{\theta_{x_0}}}:=
\sup_{x\in\R^2}\big\{\|u\|_{L^2(B^1_x)}\theta^{-1}(x-x_0)\big\}
\end{equation}
and the Sobolev spaces $H^sA^\infty$ are defined analogously.
\end{definition}
There is an important relation between weighted and uniformly local spaces which allows to get the estimates for the uniformly local norms using the corresponding weighted norms, namely,
\begin{equation}\label{0.eq2}
l\|u\|_{L^2_b}\le\sup_{x_0\in\R^2}\|u\|_{L^2_{\theta_{x_0}}}\le L\|u\|_{L^2_b}
\end{equation}
for some positive constants $l$ and $L$, see e.g. \cite{MZ1}.

Finally, the analogue of \eqref{0.eq} for Sobolev spaces is
useful  in what follows
\begin{equation}\label{0.eq.1}
l\|u\|_{H^s_{\theta_{x_0}}}^2\le
\int_{x\in\R^2}\|u\|_{H^s(B^1_{x})}^2\theta_{x_0}(x)\,dx\le
 L\|u\|_{H^s_{\theta_{x_0}}}^2.
\end{equation}

\par
The next lemma (which is a version of the Hausdorff--Young inequality)
is  helpful for proving various weighted estimates related with integral operators.
\begin{lemma}\label{Lem0.point}Let the functions $f$ and $g$ satisfy the
following point-wise estimate:
\begin{equation}\label{0.point}
|f(x)|\le C\int_{\R^2}|g(y)|\theta(x-y)\,dy,\ \ x\in\R^2.
\end{equation}
Then
\begin{equation}\label{0.l2}
\|f\|_{L^2}\le C_1\|g\|_{L^2},\ \ \|f\|_{L^2_{\theta_{x_0}}}\le C_1\|g\|_{L^2_{\theta_{x_0}}},
\end{equation}
where $C_1$ is independent of $x_0\in\R^2$. Moreover,
\begin{equation}\label{0.ll2}
\sup_{x\in\R^2}\{|f(x)|\theta^{-1}(x-x_0)\}\le C_1\sup_{x\in\R^2}\{|g(x)|\theta^{-1}(x-x_0)\}
\end{equation}
and
\begin{equation}\label{0.Ll2}
\sup_{x\in\R^2}\{\|f\|_{L^2(B^1_x)}\theta^{-1}(x-x_0)\}\le C_1\sup_{x\in\R^2}\{\|g\|_{L^2(B^1_x)}\theta^{-1}(x-x_0)\}.
\end{equation}
\end{lemma}
\begin{proof} Using \eqref{0.point}, \eqref{0.tineq} and the Fubini theorem, we have
\begin{multline}
\|f\|_{L^2}^2\le C^2\int_{x\in\R^2}\int_{y\in\R^2}\int_{z\in\R^2}|g(y)||g(z)|\theta(x-y)\theta(x-z)\,dx\,dy\,dz\le\\\le
 C_1\int_{y\in\R^2}\int_{z\in\R^2}|g(y)||g(z)|\theta(y-z)\,dy\,dz\le\\\le C_1\int_{z\in\R^2}\int_{y\in\R^2}\frac12(|g(y)|^2+|g(z)|^2)\theta(y-z)\,dy\,dz\le C_2\|g\|^2_{L^2}.
\end{multline}
To verify  the weighted estimate, we write
\begin{multline*}
\|f\|^2_{L^2_{\theta_{x_0}}}\le C\int_{x,y,z}(|g(y)|^2+|g(z)|^2)\theta(x-y)\theta(x-z)\theta(x-x_0)\,dx\,dy\,dz=\\=
2C\int_{x,y,z}|g(y)|^2\theta(x-y)\theta(x-z)\theta(x-x_0)\,dx\,dy\,dz\le\\\le
C_1\int_{x,y}|g(y)|^2\theta(x-y)\theta(x-x_0)\,dx\,dy\le \\\le 4C_1\int_y|g(y)|^2\theta(y-x_0)\int_x(\theta(x-y)+\theta(x-x_0))\,dx\,dy\le C_2\|g\|_{L^2_{\theta_{x_0}}}^2.
\end{multline*}
  Thus, estimates \eqref{0.l2} are verified. Let us check now estimate \eqref{0.ll2}, using again the key inequality \eqref{0.point}
\begin{multline}
|f(x)|\theta(x-x_0)^{-1}\le C\int_{y\in\R^2}|g(y)|\theta(y-x_0)^{-1}\theta(y-x_0)\theta(y-x)\theta(x-x_0)^{-1}\,dy\\\le 4C\sup_{y\in\R^2}\{|g(y)|\theta^{-1}(y-x_0)\}\int_{y\in\R^2}(\theta(y-x_0)+\theta(y-x))\,dy\le\\\le C_1\sup_{y\in\R^2}\{|g(y)|\theta^{-1}(y-x_0)\}
\end{multline}
and estimate \eqref{0.Ll2} is proved.
\par
Finally, on order to get \eqref{0.Ll2}, we use that
$$
\int_{y\in\R^2}|g(y)|\theta(x-y)\,dy\le
C\int_{z\in\R^2}\|g\|_{L^1(B^1_z)}\theta(x-z)\,dz,
$$
see e.g. \cite{EfZ} and after that, from \eqref{0.ll2}, we get
an even stronger estimate
$$
\sup_{x\in\R^2}\{|f(x)|\theta(x-x_0)^{-1}\}\le
C\sup_{x\in\R^2}\{\|g\|_{L^1(B^1_x)}\theta(x-x_0)^{-1}\}
$$
which finishes the proof of the lemma.
\end{proof}
We will also need in the sequel the discrete version of this lemma.
\begin{lemma}\label{Lem1.disc} Let $L$ be big enough and let
$g_j\in l^2(\Bbb Z^2)$ and $f$ satisfy
\begin{equation}\label{1.disc}
|f(y)|\le C\sum_{j\in\Bbb Z^2}|g_j|\theta(y-Lj).
\end{equation}
Then, the analogues of estimates \eqref{0.l2}, \eqref{0.ll2} and
\eqref{0.Ll2} hold uniform\-ly with respect to $L\to\infty$. Namely, we have
$$
\|f\|_{L^2}\le C\|g\|_{l^2},\ \ \|f\|_{L^2_{\theta_{x_0}}}^2\le
C\sum_{j}|g_j|^2\theta(x_0-Lj)=:C\|g\|^2_{l^2_{\theta_{x_0}}}.
$$
Moreover,
$$
\sup_{x}\{|f(x)|\theta^{-1}(x-x_0)\}\le C\sup_{j}\{|g_j|\theta^{-1}(x_0-Lj)\},
$$
where the constant $C$ is independent of $L$ and $x_0$.
\end{lemma}
Indeed, the proof of these estimates is completely analogous to the continuous case and is left to the reasder.
\par

\par
\begin{remark} The concrete choice of the kernel $\theta(z)$ in Lemmata \ref{Lem0.point} and \ref{Lem1.disc} is not essential. Important is that the kernel is integrable, so the statements of the lemmata remain true if we replace $\theta$ by $\theta^{5/6}$. We will use this fact in the sequel without further noticing.
\end{remark}
We now discuss the spaces of solenoidal vector fields and the
corresponding Leray--Helmholtz projectors. We start with the
relatively simple case where $\Omega=\R^2$.
\begin{definition}\label{Def0.sol} The space $L^2_\sigma(\R^2)$ consists of vector fields of $[L^2(\R^2)]^2$ satisfying the divergence free condition $\div u=0$, which is understood in the sense of distributions. The spaces $H^{s}_\sigma$, $H^{s}_{\theta_{x_0},\sigma}$ and $H^{s,2}_{b,\sigma}$ are defined analogously.
\end{definition}
We  denote by $P$ the orthonormal projector in $[L^2(\R^2)]^2$
to $L^2_\sigma(\R^2)$. This projector is a Zigmund--Calderon
pseudodifferential operator with the kernel
$$
K_P(z):=\frac12\delta(z)\(\begin{matrix}1&0\\0&1\end{matrix}\)+
\frac1{2\pi|z|^4}\(\begin{matrix}z_1^2-z_2^2&-2z_1z_2\\ -2z_1z_2&z_2^2-z_1^2\end{matrix}\).
$$
It is well-known that operator $P$ is bounded in $L^p$ for $1<p<\infty$,
but it is not bounded in weighted and uniformly local spaces introduced above.
However, in the equations of hydrodynamics, we often have a composition
$\nabla_x\circ P$ of the Leray projector $P$ and differentiation $\nabla_x$
and such operators have integrable cubically decaying kernels and, therefore,
are bounded in the weighted and uniformly local spaces. Namely, the following
result holds.
\begin{lemma}\label{Lem0.Pd} The operator $\Nx\circ P$ can be
presented as a sum $\Nx\circ P=N_1+N_2$, where the operator $N_1$ satisfies
$$
\|N_1 u\|_{H^{s}(B^1_{x_0})}\le C\|u\|_{H^{s+1}(B^2_{x_0})},
$$
where $C$ is independent of $x_0$ and $N_2$ is a smoothing operator with the cubically decaying kernel:
$$
|K_{N_2}(z)|\le C\theta(z).
$$
and the analogous estimates hold for all derivatives of this kernel.
\end{lemma}
Indeed, let $\Phi\in C^\infty$ be a radial cut-off function which equals to
zero in $B^{1/2}_0$ and one in $\R^2\setminus B^1_0$. Then we define the kernels
$$
K_{N_2}(z):=\nabla_z(\Phi(z)K_P(z)),\ \ K_{N_1}(z):=\nabla_z((1-\Phi(z)) K_{P}(z)).
$$
Then, the operator $N_1$ is ``local" in a sense that the value of
$(N_1u)(x_0)$ depends on the values of the function $u(x)$ for
$x\in B^1_{x_0}$ only, which gives the first estimate of the lemma.
The second kernel now does not have any singularities at zero and
differentiating the explicit formula for $K_{P}(z)$, we get the
second estimate of the lemma.
\begin{remark}\label{Rem0.Pgood} Due to the locality of the operator
$N_1$ it is bounded in all weighted and uniformly local spaces,
if $\Nx\circ P$ is bounded in the corresponding non-weighted spaces,
so the properties of $\Nx \circ P$ in weighted/uniformly local spaces
are determined by the operator $N_2$ which has cubically decaying
kernel and, therefore, satisfies the assumptions of Lemma \ref{Lem0.point}.
\end{remark}
\begin{remark}\label{Rem0.lap} Let us consider the operator $\Dx -\lambda$ with $\Ree\lambda>0$. Then, this operator realizes an isomorphism between the weighted spaces $H^{s}_{\theta_{x_0}}$ and $H^{s-2}_{\theta_{x_0}}$ as well as between $H^{s}_b$ and $L^{s-2}_b$ and also between $H^sA^\infty_{\theta_{x_0}}$ and $H^{s-2}A^\infty_{\theta_{x_0}}$  . Indeed, the kernel of this operator can be expressed explicitly in terms of Bessel functions and it is even {\it exponentially} decaying.
\end{remark}
We conclude this section by a brief discussion of the solenoidal
vector fields $L^2_\sigma(\Omega)$ and the Leray--Helmholtz
projector $P_\Omega$ in the case of bounded domains $\Omega$ (for
simplicity, we also assume that $\Omega$ is smooth and simply
connected). In this case, it is natural to define the space
$L^2_\sigma(\Omega)$ as follows
$$
L^2_\sigma(\Omega):=\bigg[\big\{u\in C^\infty_\sigma(\R^2),\,
\supp u\in\Omega\big\}\bigg]_{L^2(\Omega)},
$$
where $[\cdots]_W$ stands for the closure in the space $W$, see e.g. \cite{Temam1}.
This space has an explicit description:
$$
L^2_\sigma(\Omega)=\big\{u\in L^2(\Omega)\,: \divv u=0,\ \ u\cdot  n\big|_{\partial\Omega}=0\big\},
$$
where $ n$ is a normal vector to the boundary.

In addition,
\begin{equation}\label{projection}
P_\Omega u=-\Rot(\Dx)^{-1}_D\Rot u=-\nabla^\bot_x(\Dx)^{-1}_D\Rot u,
\end{equation}
where $(\Dx)_D$ is the Laplace operator with Dirichlet
boundary conditions, see~\cite{Temam1}. Here and below
$$
 \Nx^\bot=(\partial_{x_2},-\partial_{x_1}),\quad u^\perp=(u_2,-u_1),
 $$
 and for a vector $u$ and a scalar $\psi$
 $$
 \Rot u=\partial_{x_2}u_1-\partial_{x_1}u_2,\quad \Rot\psi=\nabla^\perp_x\psi,
 $$
so that
$$
\Rot\Rot\psi=-\Delta_x\psi.
$$

\par
Finally,  for any $u\in L^2(\R^2)$, due to the
Leray--Helmholtz decomposition (used in the whole plane
and in the domain $\Omega$), the difference $(P-P_\Omega)u$
inside  the domain $\Omega$ is a harmonic function, namely,
\begin{equation}\label{P-P}
\aligned
(P-P_\Omega)u&=\Nx Q,\quad \Dx Q=0,\\
\partial_{ n}Q\big|_{\partial\Omega}&=
(Pu-P_\Omega u)\cdot  n\vert_{\partial\Omega}=
Pu\cdot  n\vert_{\partial\Omega}.
\endaligned
\end{equation}

\section{Preliminaries II. The Vishik vortex}\label{s1}
The aim of this section is to introduce the so-called Vishik vortex, which is the key starting point for our paper, and to study its spectral properties. Namely, we assume that there exists a stationary vector field $\bar u_0\in C_0^\infty(\R^2)$ and the external force $f$ satisfying
\begin{equation}\label{1.NS-st}
\Dx \bar u_0-P\divv(\bar u_0\otimes\bar u_0)=f_0, \ \   \divv\bar u_0=0.
\end{equation}
where $P$ is the Leray projector, $u\otimes v=(u_iv_j)_{i,j=1}^2$ and
$$
\divv(U_{ij})_j=\sum_{i=1}^2\partial_{x_i}U_{ij}.
$$
For definiteness, we assume that $\supp \bar u_0,\supp f_0\subset B^1_0$ although, {\color{black} from \eqref{1.NS-st}},
in general we can guarantee only that the external force $f_0$ decays as $1/|x|^3$
as $|x|\to \infty$. However, the key property here is that $f_0\in L^1$ with all
its derivatives, which is satisfied, so we assume that $\supp f_0\subset B_0^1$
only for simplicity. {\color{black} One more useful observation which supports our simplification is that the external force can be presented in the form $f_0=P\circ \divv\tilde f_0$ where $\tilde f_0$ is fully localized.}
\par
Our key assumption on the vector field $\bar u_0$ is that it is
exponentially spectrally unstable as a solution of the
Navier--Stokes equation \eqref{1.NS-st}. The latter means that
there exists $\lambda\in\Bbb C$ such that $\Ree\lambda>0$, which
belongs to the spectrum of the linearized operator, i.e. the
operator
\begin{multline}
\Cal L_0u:=(\Delta-\lambda)u-P((\bar u_0,\Nx)u+(u,\Nx)\bar u_0)=\\=
(\Delta-\lambda)u-P\divv(\bar u_0\otimes u+u\otimes\bar u_0)
\end{multline}
satisfies $0\in\sigma(\Cal L_0)$. Here and below the spectrum is considered
in the space $L^2_\sigma(\R^2)$ of square integrable solenoidal vector fields.
\par
The existence of such velocity profiles is established in \cite{V1,V2},
see also \cite{Al} where profile with the finite support has been constructed.
Actually, the concrete form of $\bar u_0(x)$ is not important for what follows,
we only need to have the existence of an unstable eigenvalue $\lambda$
as well as a sufficiently fast decay of $\bar u_0(x)$ as $|x|\to\infty$.
\par
Note that the operator $\Cal L_0$ is Fredholm of index zero, since
it is a relatively compact perturbation of the invertible operator
$\Delta-\lambda$, and zero is an isolated eigenvalue of finite multiplicity.
In order to simplify the notations and avoid the unnecessary technicalities,
we make two inessential extra assumptions. First, we assume that $0$ is
a {\it simple} eigenvalue, this is done in order to simplify the formulas
for  the corresponding spectral projectors, avoid Jordan blocks, etc.
Actually, Vishik has constructed $\bar u_0(x)$ which satisfies this extra
assumption, but it is not essential for our method. Second, we assume
that $\lambda>0$ is real just in order to avoid adding $\Ree$ to all
scalar products and to make all functions involved real-valued. This
assumption is also inessential, but makes the formulas slightly
simpler.
\par
The next lemma gives the key properties of direct and adjoint
eigenfunctions for the operator $\Cal L_0$.
\begin{lemma}\label{Lem1. ei} The adjoint operator $\Cal L_0^*$ has the form
\begin{equation}\label{1.adj}
\Cal L_0^*v=(\Delta-\lambda)v+P\(2(\bar u_0,\Nx)v+\bar u_0^\bot\Rot v\).
\end{equation}
Let $\varphi_0$ and $\psi_0$ be the direct and adjoint eigenfunctions
 of  $\Cal L_0$, i.e.
\begin{equation}\label{1.eif}
\Cal L_0\varphi_0=\Cal L_0^*\psi_0=0,\ \ \ (\varphi_0,\psi_0)=1.
\end{equation}
Then, $\varphi_0,\psi_0\in C^\infty_\sigma(\R^2)\cap L^2_\sigma(\R^2)$,
possess the following estimates
\begin{equation}\label{1.dec}
|\varphi_0(x)|\le C(1+|x|)^{-3},\ \ \ |\psi_0(x)|\le C(1+|x|)^{-2}
\end{equation}
and the same is true for all its derivatives. Moreover, the following
representations hold
\begin{equation}\label{1.div}
\varphi_0=P\tilde\varphi_0,\ \ \psi_0=P\tilde\psi_0
\end{equation}
where the functions $\tilde\varphi_0$ and $\tilde\psi_0$ decay exponentially
as $|x|\to\infty$.
\end{lemma}
\begin{proof}
To prove~\eqref{1.adj}
 we use that for $a,b$ with $\divv a=\divv b=0$
$$
(a,\Nx b)=(b,\Nx a)-\Rot(a\times b),
$$
where $a\times b$ is a ``vertical'' vector $ a\times b=a\cdot
b^\perp=-a^\perp\cdot b $, and that
$$
(\Rot a,b)=(a,\Rot b).
$$
Using  also that
$
((a,\Nx b),c)=-((a,\Nx c),b)
$
we obtain 
$$
\aligned
&((\bar u_0,\Nx)u,v)=-((\bar u_0,\Nx)v,u),\\
&((u,\Nx)\bar u_0,v)=((\bar u_0,\Nx)u,v)-(\Rot( u\times\bar u_0),v)\\=
&-((\bar u_0,\Nx)v,u)-(u,\bar u_0^\perp\Rot v),
\endaligned
$$
which gives~\eqref{1.adj}.

The smoothness of eigenfunctions follows from the smoothness of
$\bar u_0$ and the elliptic regularity. In order to verify
\eqref{1.dec} and \eqref{1.div}, we rewrite the equation for
$\varphi_0$ as follows:
\begin{equation}\label{fi0}
\varphi_0=(\Dx-\lambda)^{-1}P\divv(\bar u_0\otimes\varphi_0+\varphi_0\otimes\bar u_0)
\end{equation}
Since the function $\bar u_0\otimes\varphi_0+\varphi_0\otimes\bar u_0$ is localized and operators $P\divv$ and $(\Dx-\lambda)^{-1}$ have cubically and exponentially decaying kernels, this gives the cubic decay of the eigenfunction $\varphi_0$. Moreover, \eqref{1.div} also holds if we take $\tilde\varphi_0:=(\Dx-\lambda)^{-1}\divv(\bar u_0\otimes\varphi_0+\varphi_0\otimes\bar u_0)$.   Analogously, for the adjoint eigenfunction $\psi_0$, we have
$$
\psi_0=-(\Dx-\lambda)^{-1}P\(2(\bar u_0,\Nx)\psi_0+\bar u_0^\bot\Rot \psi_0\),
$$
so, in contrast to the direct eigenfunction, we have only the operator $P$ (without derivatives) which has quadratically decaying kernel only and, therefore, we are able to verify only the quadratic decay rate of the adjoint eigenfunctions. Finally, fixing $\tilde\psi_0:=-(\Dx-\lambda)^{-1}\(2(\bar u_0,\Nx)\psi_0+\bar u_0^\bot\Rot \psi_0\)$, we get \eqref{1.div} and finish the proof of the lemma.
\end{proof}
We now introduce the spectral projector
$$
\Pi_0w:=(w,\psi_0)\varphi_0=(w,\tilde\psi_0)\varphi_0,\ \ w\in L^2_\sigma(\R^2)
$$
and consider the auxiliary equation
\begin{equation}\label{1.au}
\Cal L_0w+\eb\Pi_0w=g.
\end{equation}
Then, in contrast to the main equation $\Cal L_0w=g$, this equation is solvable for all $g\in L^2_\sigma(\R^2)$ if $\eb\ne0$ and if we assume that $(g,\psi_0)=0$, which is the solvability condition for the main equation, then multiplication of \eqref{1.au} by $\psi_0$ gives $(w,\psi_0)=0$ and $w$ solves the main equation as well. For many reasons, it is more convenient for us to develop the perturbation theory for auxiliary equation \eqref{1.au}, which is uniquely solvable. The next lemma, which gives weighted estimates for the solution of \eqref{1.au} is one of the key technical tools for what follows.
\begin{lemma}\label{Lem1.we} Let $g\in L^2_\sigma(\R^2)$ satisfy $\supp g\subset B_0^1$. Then the solution $w$ of problem \eqref{1.au} possesses the following estimate:
\begin{equation}\label{1.reg1}
\|w\|_{H^2(B^1_y)}\le C\|g\|_{L^2}\theta(y),
\end{equation}
where $\theta(x):=(1+|x|)^{-3}$ and the constant $C$ depends on $\eb>0$, but is independent of $y\in\R^2$.
\par
Assume that now $g=P\divv G$, where $\supp G\subset B^1_0$. Then the analogue of estimate \eqref{1.reg1} holds:
\begin{equation}\label{1.reg2}
\|w\|_{H^2(B^1_y)}\le C\|G\|_{H^1}\theta(y).
\end{equation}
\end{lemma}
\begin{proof} Indeed, the unique solvability of equation \eqref{1.au} is immediate and the non-weighted $H^2$-estimate follows from the elliptic regularity, so we have
$$
\|w\|_{H^2}\le C\|g\|_{L^2}
$$
and the analogous estimate for the second case. We now write equation \eqref{1.au} in the form of
$$
w=(\Delta-\lambda)^{-1}(g+P\div(\bar u_0\otimes w+w\otimes \bar u_0)-\eb\Pi_0w).
$$
Since $\bar u_0$ has a compact support, the function
$$
F:=g+P\div(\bar u_0\otimes w+w\otimes \bar u_0)-\eb\Pi_0w
$$
 has the cubic decay rate, i.e.
$$
\|F\|_{L^2(B^1_y)}\le C\|g\|_{L^2}\theta(y),\ \ y\in\R^2
$$
and, therefore, the weighted regularity for the operator
$\Dx-\lambda$ gives the desired estimate \eqref{1.reg1}.
Estimate \eqref{1.reg2} can be verified analogously.
This finishes the proof of the lemma.
\end{proof}
\begin{remark}\label{Rem1.loc} Arguing analogously,
it is not difficult to establish slightly more general estimates
replacing the condition $\supp g\subset B^1_0$ by $\supp g\subset
B^1_{x_0}$. Then the function $\theta(y)$ in the right-hand side of
\eqref{1.reg1} should be replaced by $\theta(x_0-y)$, namely,
\begin{equation}\label{1.nnew}
\|w\|_{H^2(B^1_y)}\le
C\|g\|_{L^2}\theta(x_0-y).
\end{equation}
Indeed, to
prove this generalization, we first solve equation
$$
(\Dx-\lambda)V=g
$$
and due to the weighted regularity, obtain that
$$
\|V\|_{H^2(B^1_y)}\le C\|g\|_{L^2}\theta(y-x_0).
$$
After that, for the new function $W=w-V$, we have
$$
\Cal L_0W+\eb\Pi_0W=-\eb\Pi_0V+P\divv(\bar u_0\otimes V+V\otimes\bar u_0)=:\tilde g.
$$
Then
we get
$$
\|W\|_{H^2}\le C\|\tilde g\|_{L^2}\le  C\|g\|_{L^2}\theta(x_0).
$$
In fact, to see that the  inequality $\|\tilde g\|_{L^2}\le  C\|g\|_{L^2}\theta(x_0)$
holds we observe that $\Pi_0V=(V,\tilde\psi_0)\varphi_0$, since $\divv V=0$,
and therefore
\begin{multline*}
\|\Pi_0V\|_{L^2}\le \|\varphi_0\|_{L^2}\int_{\mathbb R^2}|V(y)||\tilde\psi_0(y)|dy\\
\le C\|g\|_{L^2}\int_{\mathbb R^2}\theta(y-x_0)\theta(y)dy
\le C'\|g\|_{L^2}\theta(x_0),
\end{multline*}
where we used the embedding $H^2(B^1_y)\hookrightarrow C(B^1_y)$ and
\eqref{0.tineq}.
\par
For the second term $U:=P\divv(\bar u_0\otimes V+V\otimes\bar u_0)$
we write
$$
|U(x)|\le C \int_{\mathbb R^2}\theta(x-y)|\bar u_0(y)||V(y)|dy
$$
and by Lemma~\ref{Lem0.point}, taking into account that
$\supp\bar u_0\subset B_0^1$
we obtain
$$
\|U\|_{L^2}\le C\|\bar u_0V\|_{L^2}\le C'\|V\|_{L^2(B_0^1)}\le
C_1\|g\|_{L^2}\theta(x_0).
$$
Arguing now as in the proof of the lemma, we infer that
$$
\|W\|_{H^2(B^1_y)}\le C\|g\|_{L^2}\theta(y)\theta(x_0)\le C\|g\|_{L^2}\theta(x_0-y).
$$
So, estimate \eqref{1.nnew} is proved. This estimate is important for establishing the spectral properties
of multi-vortices in weighted spaces.
\end{remark}
Let us now cut off the eigenfunctions $\varphi_0$ and $\psi_0$ as follows.
Let $L>0$ be a big number which will be fixed below, let
$0<\alpha<\beta<1$ be two fixed numbers and let $\Phi_L(x):=\Phi(x/L)$
 be a smooth function such that $\Phi(x)=1$ for $|x|\le\alpha$
 and $\Phi(x)=0$ for $|x|>\beta$. Finally, let
\begin{equation}
\tilde\psi_{0,L}:=\Phi_L\tilde\psi_0,\ \ \varphi_{0,L}:=\Nx^\bot(\Phi_L S_0),
\end{equation}
  where
  \begin{equation}\label{1.s0}
  S_0:=-(\Dx-\lambda)^{-1}\Dx^{-1}\Rot P\div(\varphi_0\otimes\bar u_0+\bar u_0\otimes\varphi_0)
  \end{equation}
is the stream function associated with the eigenfunction $\varphi_0$,
see \eqref{fi0} and \eqref{projection} with $\Omega=\mathbb R^2$.
The properties of these modified spectral functions are collected in
the following lemma.
\begin{lemma}\label{Lem1.cut} The functions $\varphi_{0,L},\tilde\psi_{0,L}$
and $\psi_{0,L}:=P\tilde\psi_{0,L}$ satisfy
\begin{equation}\label{1.uni}
|\varphi_{0,L}(x)|\le C\theta(x), \ \
|\tilde\psi_{0,L}(x)|\le Ce^{-\gamma|x|},\
\ |\psi_{0,L}(x)|\le C(1+|x|)^{-2}
\end{equation}
uniformly with respect to $L$ (i.e. positive constants $C$ and $\gamma$ are
independent of $L$) and the same is true for their derivatives. Moreover, the
following estimates hold:
\begin{equation}\label{1.cl}
|\varphi_0(x)-\varphi_{0,L}(x)|\le CL^{-1/2}\theta(x)^{-5/6},\ \
|\tilde\psi_0(x)-\tilde\psi_{0,L}(x)|\le Ce^{-\gamma|x|/2}
\end{equation}
and the same is true for the derivatives. Finally
\begin{equation}\label{1.app}
\|\Cal L_0\varphi_{0,L}\|_{L^2}+\|\Cal L_0^*\psi_{0,L}\|_{L^2}\le CL^{-1/2}.
\end{equation}
The $H^2$ analogue of the last estimate holds.
\end{lemma}
\begin{proof} Indeed, the second estimate of \eqref{1.uni} is obvious
and the third one follows from the second one. Let us verify the first one.
By the definition of $\varphi_{0,L}$
$$
\varphi_{0,L}=\Phi_L\varphi_0+\nabla^\bot\Phi_L S_0.
$$
For the first term in the right-hand side we have the desired estimate
since $\varphi_0$ has the cubic decay rate. For the second term, $S_0$ satisfies
\begin{equation}\label{1.sest}
|S_0(x)|\le C(1+|x|)^{-2}.
\end{equation}
Indeed, this follows from \eqref{1.s0} and from the fact
that $\Dx^{-1}\Rot P\divv$ is a Zygmund--Calderon operator
and therefore has a quadratically decaying kernel. Finally,
since $\nabla\Phi_L$ is of order $L^{-1}$ and $\nabla^\bot\Phi_L(x)$ is
non-zero only if $|x|\sim L$, we may write
\begin{equation}\label{1.pL}
|\nabla\Phi_L(x)|\le CL^{-1}\sim C(1+|x|)^{-1}\sim L^{-1/2}(1+|x|)^{-1/2}
\end{equation}
and this finishes the proof of \eqref{1.uni}.
\par
Let us verify the closeness estimates \eqref{1.cl}. Again, the second one is
obvious and for the first one we write
$$
\varphi_0(x)-\varphi_{0,L}(x)=(1-\Phi_L)\varphi_0(x)-\nabla^\perp\Phi_L(x) S_0(x).
$$
The desired estimate for the second term in the right-hand side is already
obtained and for the first term we use that $1-\Phi_L$ vanishes for $|x|<\alpha L$,
so
$$
|1-\Phi_L(x)|\le (1+|x|)^{-1/2}(1+|x|)^{1/2}\le CL^{-1/2}(1+|x|)^{1/2}
$$
and this finishes the proof of \eqref{1.cl}.
\par
Finally, estimates \eqref{1.app} follow from the fact that
operators $\Cal L_0$ and $\Cal L_0^*$ are bounded as operators
from $H^2$ to $L^2_\sigma$, estimates \eqref{1.cl} and the fact
that $P$ is bounded from $H^2$ to $H^2$. The $H^2$-analogue follows
from the fact that the modified eigenfunctions are smooth and their
derivatives also satisfy the decaying estimates.
This finishes the proof of the lemma.
\end{proof}
\begin{remark}\label{Rem1.app-loc} Arguing as in Remark \ref{Rem1.loc}, we may
localize the first estimate of \eqref{1.app}, namely,
\begin{equation}\label{1.ap-loc}
\|\Cal L_0\varphi_{0,L}\|_{H^s(B^1_x)}\le C_sL^{-1/2}\theta^{5/6}(x).
\end{equation}
for all $s\ge0$. Crucial for us is that the weight $\theta^{5/6}$ is still
integrable, so no problems with estimates in uniformly local spaces as with
summing infinitely many spatially localized terms like this  will arise.
\par
In contrast to this, surprisingly, the structure of the adjoint operator
$\Cal L_0^*$ is worse since it contains the operator $P$ which is not
compensated by extra differentiation, so we do not know whether this
operator is well-posed in uniformly local spaces and this produces extra
difficulties to fight with.
\end{remark}
To conclude this section, we introduce the localized projector
\begin{equation}\label{1.ppL}
\Pi_{0,L}u:=\alpha_L(u,\tilde\psi_{0,L})\varphi_{0,L},\ \
\alpha_L:=(\tilde\psi_{0,L},\varphi_{0,L})^{-1}.
\end{equation}
\begin{corollary}\label{Cor1.dpr} The projectors $\Pi_0$ and $\Pi_{0,L}$ satisfy
\begin{equation}\label{1.dpr}
\|\Pi_{0,L}u-\Pi_0u\|_{L^2}\le CL^{-1/2}\|u\|_{L^2},
\end{equation}
where the constant $C$ is independent of $L$.
\end{corollary}
Indeed, this estimate is an immediate corollary of \eqref{1.uni} and \eqref{1.cl}.
\par
We introduce now the localized version of equation \eqref{1.au}
\begin{equation}\label{1.auL}
\Cal L_0w+\eb \Pi_{0,L}=g.
\end{equation}
Then, the following result holds.
\begin{corollary}\label{Lem1.auL} For every $\eb>0$ there exists $L_0$
such that problem \eqref{1.auL} is uniquely solvable for every $L>L_0$ and
\begin{equation}
\|w\|_{H^2}\le C\|g\|_{L^2},
\end{equation}
where $C$ is independent of $L$. Moreover the  analogues of \eqref{1.reg1} and
\eqref{1.reg2} as well as estimates from Remark \ref{Rem1.loc} hold.
\end{corollary}
Indeed, the unique solvability in the non-weighted spaces follows
from the analogous result for equation \eqref{1.au}, estimate \eqref{1.dpr}
by perturbation arguments. The weighted estimates can be proved exactly
as in Lemma~\ref{Lem1.we}.

\section{Multi-vortices and solvability of the auxiliary equation:
the case of the whole plane}\label{s2}
In this section, we introduce the multi-vortex profile for the case
of $\Omega=\R^2$ and verify the unique solvability of the corresponding
auxiliary equation. We first take a big number $L$ and introduce the
finite or infinite set $\Xi\subset L\mathbb Z^2$ of vortex centers.
Let $\bar u_0$ be a Vishik vortex satisfying all properties stated in
the previous section. Then, the multi-vortex profile is defined as
follows
\begin{equation}\label{2.m}
\bar u_\Xi(x):=\sum_{\xi_j\in\Xi}\bar u_j(x),\ \ \bar u_j(x):=\bar u_0(x-\xi_j).
\end{equation}
Since the supports of $\bar u_j$ do not intersect, this profile
satisfies the stationary Navier--Stokes equation
\begin{equation}\label{2.ns}
-\Dx \bar u_\Xi+P\divv(\bar u_\Xi\otimes\bar u_\Xi)=f_\Xi:=\sum_{\xi_j\in\Xi}f_j,
 \ \ f_j(x):=f_0(x-\xi_j).
\end{equation}
We consider the linearization of the NS equations on
this velocity profile and introduce the following operator
\begin{equation}\label{2.lin}
\Cal L_\Xi w:=(\Dx-\lambda)w-P\divv(\bar u_\Xi\otimes w+w\otimes\bar u_\Xi).
\end{equation}
Our main idea is to prove that this operator has at least $\#\Xi$ eigenvalues
near zero if $L$ is large enough. To this end, following the general strategy
of \cite{MiZ}, we introduce the projector
\begin{equation}\label{2.pro}
\Pi_{\Xi,L}w:=\alpha_L\sum_{\xi_j\in\Xi}(w,\tilde\psi_{j,L})\varphi_{j,L},
\end{equation}
where $\tilde\psi_{j,L}(x):=\tilde\psi_{0,L}(x-\xi_j)$ and
$\varphi_{j,L}(x):=\varphi_{0,L}(x-\xi_j)$, and consider the auxiliary equation
\begin{equation}\label{2.au}
\Cal L_{\Xi}w+\eb\Pi_{\Xi,L}w=g.
\end{equation}
The main result of this section is the following theorem.
\begin{theorem}\label{Th2.main} Let $\eb>0$ be small enough. Then, for
sufficiently big $L$,  problem \eqref{2.au} is uniquely solvable for any
$g\in L^2_\sigma(\R^2)$ and the following estimate holds
\begin{equation}\label{2.main}
\|w\|_{H^2}\le C\|g\|_{L^2}
\end{equation}
uniformly with respect to $L\to\infty$. Moreover,
the analogous estimate holds in $L^2_{b,\sigma}(\R^2)$,
 $L^2_{\theta_{x_0},\sigma}$ and $L^2A^\infty_{\theta^{5/6}}(\R^2)$ as well.
\end{theorem}
\begin{proof} We split the proof in a number of lemmata.
\begin{lemma}\label{Lem2.pro} The operator $\Pi_{\Xi,L}$ is a bounded
operator in $L^2_\sigma$. Moreover. the following estimate holds:
\begin{equation}\label{2.pl2}
\|\Pi_{\Xi,L}w\|_{L^2}\le C\|w\|_{L^2}
\end{equation}
and the constant $C$ is independent of $L$ and $\Xi$. The analogous
estimates for the other spaces mentioned in the theorem also hold.
\end{lemma}
\begin{proof} Indeed, according to \eqref{1.uni},
\begin{multline}
|(\Pi_{\Xi,L}w)(x)|\le C\int_{\R^2}\sum_{\xi_j\in\Xi}|w(y)|\theta(y-\xi_j)\theta(x-\xi_j)\,dy\le\\\le
4C\int_{\R^2}|w(y)|\theta(x-y)\sum_{\xi_j\in\Xi}(\theta(x-\xi_j)+\theta(y-\xi_j))\,dy\le\\\le C'\int_{\R^2}|w(y)|\theta(x-y)\,dy
\end{multline}
and Lemma \ref{0.point} gives us the desired estimate and finishes the proof of the lemma.
\end{proof}
\begin{corollary}\label{Cor2.lin} For a sufficiently small $\eb>0$, the equation
\begin{equation}\label{2.elin}
(\Dx-\lambda)W+\eb\Pi_{\Xi,L}W=g
\end{equation}
is uniquely solvable for every $g\in L^2_\sigma$ and the following estimate holds:
\begin{equation}\label{2.linest}
\|W\|_{H^2}\le C\|g\|_{L^2},
\end{equation}
where $C$ is independent of $L$. Also, the analogous estimates in the
other spaces mentioned in the theorem hold.
\end{corollary}
Indeed, the result holds for $\eb=0$ due to the elliptic regularity and
its validity for small $\eb>0$ follows by perturbation arguments.
\par
From now on we fix $\eb>0$ in such a way that the assertion of the corollary holds.
\par
As the next step, we construct an approximate solution for the problem
\begin{equation}\label{2.aau}
(\Dx-\lambda)u-P\divv(\bar u_\Xi\otimes u+u\otimes\bar
u_\Xi)+\eb\Pi_{\Xi,L}u=P\divv(\tilde g),
\end{equation}
where $\tilde g=\sum_{\xi_j\in\Xi}\tilde g_j$ with
$\supp \tilde g_j\subset B^1_{\xi_j}$. Namely, we define $u_j$ as a
unique solution of the one-vortex problem
\begin{equation}\label{2.one}
(\Dx-\lambda)u_j-P\divv(\bar u_j\otimes u_j+u_j\otimes\bar u_j)+
\eb\Pi_{j,L}u_j=P\divv\tilde g_j
\end{equation}
and take
$$
\tilde u:=\sum_{\xi_j\in\Xi}u_j.
$$
Note that, due to the shifted version of estimate \eqref{1.reg2} we have
\begin{equation}\label{2.uj}
\|u_j\|_{H^2(B^1_y)}\le C\|\tilde g_j\|_{H^1}\theta(y-\xi_j),
\end{equation}
where the constant $C$ is independent of $y\in\R^2$ and $\xi_j\in\Xi$.
\end{proof}
\begin{lemma}\label{Lem2.app} The approximate solution $\tilde u$ satisfies
\begin{equation}\label{2.app-est}
\|\tilde u\|_{H^2}\le C\|\tilde g\|_{H^1},
\end{equation}
where the constant $C$ is independent of $L$. Moreover, the analogous
estimates hold in the spaces $L^2_b$, $L^2_{\theta_{x_0}}$ and
$L^2A^\infty_{\theta_{x_0}^{5/6}}$.
\end{lemma}
\begin{proof} Using \eqref{2.uj}, we get
$$
\|\tilde u\|_{H^2(B^1_x)}\le C\sum_{\xi_j\in\Xi}\|\tilde g_j\|_{H^1}\theta(x-\xi_j)
$$
and applying Lemma \ref{Lem1.disc}, we arrive at
$$
\|\tilde u\|^2_{H^2}\le C\int_{x\in\R^2}\|\tilde u\|^2_{H^2(B^1_x)}\,dx
\le C_1\sum_j\|g_j\|^2_{H^1}=C_1\|\tilde g\|^2_{H^1},
$$
where we have used that the supports of $\tilde g_j$ do not intersect. Analogously,
$$
\|\tilde u\|^2_{H^2_{\theta_{x_0}}}\le
C\sum_{i\in\Xi}\|\tilde g_i\|^2_{H^1}\theta(\xi_i-x_0)\le
C_1\|\tilde g\|^2_{H^1_{\theta_{x_0}}}.
$$
The analogous estimate in $L^2_b$ follows from the last estimate via \eqref{0.eq2}.
Finally, the same Lemma \ref{Lem1.disc}, where $\theta$ is replaced
by $\theta^{5/6}$ gives
$$
\|\tilde u\|_{H^2A^\infty_{\theta_{x_0}^{5/6}}}\le
C\sup_{\xi_i\in\Xi}\{\|\tilde g_i\|_{H^1}\theta(\xi_i-x_0)^{-5/6}\}\le
C\|\tilde g\|_{H^1A^\infty_{\theta^{5/6}_{x_0}}}.
$$
 Thus, the lemma is proved.
\end{proof}
The approximate solution $\tilde u$ solves the equation
\begin{equation}\label{2.aaut}
(\Dx-\lambda)\tilde u-P\divv(\bar u_\Xi\otimes \tilde u+\tilde u\otimes\bar u_\Xi)+\eb\Pi_{\Xi,L}\tilde u=P\divv(\tilde R)+\hat R,
\end{equation}
where
$$
\tilde R:=-\sum_{i,j\,: i\ne j}(\bar u_i\otimes u_j+u_j\otimes\bar u_i),\ \
\hat R:=\eb\alpha_L\sum_{i,j\,: i\ne j}(u_i,\tilde\psi_{j,L})\varphi_{j,L}.
$$
The next  lemma gives the smallness of the approximation error.
\begin{lemma}\label{Lem2.e-small} The functions $\hat R$ and $\tilde R$ satisfy
\begin{equation}\label{2.small}
\|\tilde R\|_{H^1}+\|\hat R\|_{L^2}\le CL^{-3}\|\tilde g\|_{H^1},
\end{equation}
where the constant $C$ is independent of $L$. Moreover, the analogous
estimates hold in $L^2_b$, $L^2_{\theta_{x_0}}$ and
$L^2A^\infty_{\theta_{x_0}^{5/6}}$.
\end{lemma}
\begin{proof}  We start with the estimate for $\hat R$. Using \eqref{0.tineq}, \eqref{0.sum-small}, \eqref{1.uni} and \eqref{2.uj}, we have
\begin{multline}\label{2.ter}
|\hat R(x)|\le C\int_{y}\sum_{i,j;\, i\ne j}\|\tilde g_i\|_{H^1}\theta(y-\xi_i)\theta(y-\xi_j)\,dy\theta(x-\xi_j)\le \\\le 4C\sum_{i,j;\,i\ne j}\|\tilde g_i\|_{H^1}\theta(\xi_i-\xi_j)\theta(x-\xi_j)\int_y(\theta(y-\xi_i)+\theta(y-\xi_j))\,dy\le\\\le
C\sum_{i,j:\,i\ne j}\|\tilde g_i\|_{H^1}\theta(x-\xi_j)\theta(\xi_i-\xi_j).
\end{multline}
\begin{color}{black}
Therefore,
\begin{multline}\label{2.simplified}
|\hat R(x)|^2\le C\sum_{i,j\,i\ne j}\sum_{k,l\,k\ne l}\|\tilde g_i\|_{H^1}\|\tilde g_k\|_{H^1}\theta(x-\xi_j)\theta(x-\xi_l)\theta(\xi_i-\xi_j)\theta(\xi_k-\xi_l)\le\\\le
C\sum_{i,j\,i\ne j}\sum_{k,l\,k\ne l}\frac12(\|\tilde g_i\|_{H^1}^2+\|\tilde g_k\|_{H^1}^2)\theta(x-\xi_j)\theta(x-\xi_l)\theta(\xi_i-\xi_j)\theta(\xi_k-\xi_l)\le\\\le
C\sum_{i,j\,i\ne j}\sum_{k,l\,k\ne l}\|\tilde g_i\|_{H^1}^2\theta(x-\xi_j)\theta(x-\xi_l)\theta(\xi_i-\xi_j)\theta(\xi_k-\xi_l)\le\\\le
CL^{-3}\sum_{i,j\,i\ne j}||\tilde g_i\|_{H^1}^2\theta(x-\xi_j)\theta(\xi_i-\xi_j),
\end{multline}
where we have implicitly used that the terms containing $\{\tilde g_i\}_{i\in\mathbb Z}$ and $\{\tilde g_k\}_{k\in\mathbb Z}$ are identical up to change $\{i,j\}\to\{k,l\}$ followed by summation first with respect to $k$ and then with respect to $l$. Integrating with respect to $x$ followed by taking a sum with respect to $j$, we arrive at
\end{color}
\begin{equation}
\|\hat R\|_{L^2}^2\le CL^{-6}\sum_{i}\|\tilde g_i\|_{H^1}^2  \le CL^{-6}\|\tilde g\|_{H^1}^2
\end{equation}
and this, in turn,  gives the desired estimate
for $\hat R$ in $L^2$.

\begin{color}{black}
 Let us prove the estimate for $L^2_{\theta_{x_0}}$-norm. According to \eqref{2.simplified},
 \begin{multline}
 |\hat R(x)|^2\theta(x-x_0)\le CL^{-3}\sum_{i,j}\|\tilde g_i\|^2_{H^1}\theta(x-x_0)\theta(x-\xi_j)\theta(\xi_j-\xi_i)\le\\\le CL^{-3}\sum_{i,j}\|\tilde g_i\|^2_{H^1}\theta(x-x_0)\theta(x-\xi_i)[\theta(\xi_j-\xi_i)+\theta(x-\xi_j)]\le\\\le CL^{-3} \sum_{i}\|\tilde g_i\|^2_{H^1}\theta(x-x_0)\theta(x-\xi_i)
 \le\\\le CL^{-3}\sum_i\|\tilde g_i\|^2_{H^1}\theta(x_0-\xi_i)[\theta(x-\xi_i)+\theta(x-x_0)].
\end{multline}
Integrating the last inequality with respect to $x$, we arrive at
\end{color}
$$
\|\hat R\|^2_{L^2_{\theta_{x_0}}}\le C_2L^{-3}
\sum_{i}\|\tilde g_i\|^2_{H^1}\theta(\xi_i-x_0)
\le C_3L^{-3}\|\tilde g\|^2_{H^1_{\theta_{x_0}}},
$$
and the desired estimate in $L^2_{\theta_{x_0}}$ is proved.
The estimate for the $L^2_b$-norm now follows from \eqref{0.eq2}.

\begin{color}{black}
 Let us now estimate the $L^2A^\infty_{\theta_{x_0}^{5/6}}$-norm of $\hat R$. Indeed,
\begin{multline}
|\hat R(x)|\theta^{-5/6}(x-x_0)\le\\ C\sup_i\{\|\tilde g_i\|_{H^1}\theta(\xi_i-x_0)^{-5/6}\}
\sum_{i,j,\,i\ne j}\theta(x-\xi_j)\theta(\xi_i-\xi_j)
\theta(\xi_i-x_0)^{5/6}\theta(x-x_0)^{-5/6}\le\\\le
C_1\|\tilde g\|_{H^1A^\infty_{\theta_{x_0}^{5/6}}}\sum_{i,j,\,i\ne j}\theta(x-\xi_j)\theta(\xi_i-\xi_j)
\theta(\xi_i-x_0)^{5/6}\theta(x-x_0)^{-5/6}.
\end{multline}
Using inequality \eqref{0.tineq} several times, we estimate summand in the last formula as follows:
\begin{multline}
\theta(x-\xi_j)\theta(\xi_i-\xi_j)
\theta(\xi_i-x_0)^{5/6}\theta(x-x_0)^{-5/6}= [\theta(\xi_i-\xi_j)\theta(x-\xi_j)]^{1/6}\times\\\times[\theta(\xi_i-\xi_j)\theta(x-\xi_j)
\theta(\xi_i-x_0)]^{5/6}\theta(x-x_0)^{-5/6}\le C[\theta(\xi_i-\xi_j)\theta(x-\xi_j)]^{1/6}\times\\\times
[\(\theta(\xi_i-x_0)+\theta(\xi_i-\xi_j)\)\(\theta(\xi_j-x_0)+\theta(x-\xi_j)\)]^{5/6}.
\end{multline}
Applying now the H\"older inequality to the sum considered, we arrive at
\begin{multline}
\sum_{i,j\,i\ne j}\theta(x-\xi_j)\theta(\xi_i-\xi_j)
\theta(\xi_i-x_0)^{5/6}\theta(x-x_0)^{-5/6}\le\\\le C\(\sum_{i,j\,i\ne j}\theta(\xi_i-\xi_j)\theta(x-\xi_j)\)^{1/6}\times\\\times\(\sum_{i,j\,i\ne j}\(\theta(\xi_i-x_0)+\theta(\xi_i-\xi_j)\)\(\theta(\xi_j-x_0)+\theta(x-\xi_j)\)\)^{5/6}\!\!\!\!\!\!.
\end{multline}
In both sums in the RHS of the last inequality we make summation first by $i\ne j$ and then by $j$. The first term gives us the multiplier $CL^{-1/2}$ and the second one will be bounded as $L\to\infty$. Combining everything together we end up with the desired estimate
$$
|\hat R(x)|\theta(x-x_0)^{-5/6}\le CL^{-1/2}\|\tilde g\|_{H^1A^\infty_{\theta_{x_0}^{5/6}}}
$$
and, therefore,
\begin{multline}
\|\hat R\|_{L^2A^\infty_{\theta_{x_0}^{5/6}}}=
\sup_x\{\|\hat R\|_{L^2(B^1_x)}\theta(x-x_0)^{-5/6}\}\le\\\le
 C\sup_x\{|\hat R(x)|\theta(x-x_0)^{-5/6}\}\le C L^{-1/2}\|\tilde g\|_{H^1A^{\infty}_{\theta_{x_0}^{5/6}}}.
\end{multline}
\end{color}
 This finishes the estimates for $\hat R$.
\par
Let us consider now the function $\tilde R$. Using that $\bar u_i$ is localized, we have
\begin{multline}
\|\tilde R\|_{H^1(B^1_x)}\le C\sum_{i,j;\,i\ne j}\|u_j\|_{H^1(B^1_{\xi_i})}\theta(x-\xi_i)\le\\\le
C\sum_{i,j:j\ne i}\|\tilde g_j\|_{H^1}\theta(\xi_i-\xi_j)\theta(x-\xi_i).
\end{multline}


This estimate is almost identical to \eqref{2.ter}
for the function $\hat R$, so, arguing as before, we derive the
analogous estimates for $\tilde u$ and finish the proof of the lemma.
\end{proof}
\begin{remark} Arguing analogously, we may check the smallness
of $\hat R$ and $\tilde R$ in the spaces
$L^2A^\infty_{\theta_{x_0}^{1-\kappa}}$
for any (small) $\kappa>0$, but we do not know how to get smallness for
the case $\kappa=0$.
\end{remark}
We are ready to verify the existence of a solution for equation \eqref{2.au}.
\begin{lemma}\label{Lem2.ex} For sufficiently small $\eb>0$
there exists a solution $w$ of problem \eqref{2.au} and this
solution satisfies estimate \eqref{2.main} uniformly with respect
to $L\to\infty$.  Moreover the analogous estimate holds in the spaces
$L^2_b$, $L^2_{\theta_{x_0}}$ and $L^2A^\infty_{\theta_{x_0}^{5/6}}$.
\end{lemma}
\begin{proof} We give the proof for the $L^2$-case only, the remaining
 ones are completely analogous. Let us write out the right-hand side $g$
 of equation \eqref{2.au} in the form
$$
g=P\divv \tilde g+\hat g,
$$
where $\tilde g=\sum_{\xi_j}g_j$ and $\supp g_j\subset B^1_{\xi_j}$.
For instance, at the first step, we may take $\tilde g=0$. Then, using
Corollary \ref{Cor2.lin}, we solve equation \eqref{2.elin} with the
right-hand side $\hat g$. Let $W$ be the corresponding solution.
Then, $W$ satisfies \eqref{2.linest} and the function $v=u-W$ solves
equation \eqref{2.aau} where $\tilde g$ is replaced by
$$
\tilde G:=\tilde g-\sum_{\xi_i\in\Xi}(\bar u_i\otimes W+W\otimes\bar u_i).
$$
Since $\bar u_i$ is smooth and localized, the function $\tilde G$
has the same structure as $\tilde g$, so we may construct the
approximate solution $\tilde u$ of equation \eqref{2.aau},
which satisfy equation \eqref{2.aaut} and the error functions
$\tilde R$ and $\hat R$ satisfy
\begin{equation}\label{2.Rsmall}
\|\tilde R\|_{H^1}+\|\hat R\|_{L^2}\le CL^{-3}\|\tilde G\|_{H^1}\le
C_1L^{-3}\(\|\tilde g\|_{H^1}+\|\hat g\|_{L^2}\).
\end{equation}
We also have the estimate
\begin{equation}\label{2.bound}
\|W\|_{H^2}+\|\tilde u\|_{H^2}\le C\(\|\tilde g\|_{H^1}+\|\hat g\|_{L^2}\).
\end{equation}
It remains to note that equation \eqref{2.aaut} has the same structure as
\eqref{2.au} if we take $\tilde g_{new}:=\tilde R$, $\hat g_{new}:=\hat R$
and $g_{new}:=P\divv \tilde g_{new}+\hat g_{new}$, so we may iterate the
construction of the approximate solution. The convergence of the error to
zero is guaranteed by estimate \eqref{2.Rsmall} and the fact that the
constructed solution belongs to $H^2$ follows from estimate \eqref{2.bound}.
This finishes the proof of the lemma.
\end{proof}
Thus, to prove the theorem, we only need to verify the uniqueness. This is
immediate if $\#\Xi$ is finite since the corresponding operator is Fredholm
of index zero, but this argument does not work in the case of infinitely many
vortices. Although in the present paper, we are interested in the cases where
this number is finite or we have periodicity assumption and in both cases the
Fredholm argument works, for completeness, we give the proof of uniqueness in
$L^2_b$-case. The case of the space $L^2_{\theta_{x_0}}$ is analogous, but
technically a bit more difficult since we should care about the proper cut
off for the divergence free functions from this space and cannot use the
stream function and Zygmund--Calderon theory directly, so we left this case
to the reader.
\begin{lemma}\label{Lem2.unique} Equation \eqref{2.au} is uniquely solvable
in $H^2_b$ for any $g\in L^2_{b,\sigma}$.
\end{lemma}
\begin{proof} Assume that it is not true. Then, there exists a function $Z\in H^2_b$ such that
\begin{equation}\label{2.aue}
\Cal L_{\Xi}Z+\eb\Pi_{\Xi,L}Z=0.
\end{equation}
Then, due to the elliptic regularity, $Z\in C^\infty_b(\R^2)$ and
$$
\|Z\|_{C^\infty}\le C\|Z\|_{L^2_b}
$$
Moreover,
\begin{multline}\label{2.Z}
Z=(\Dx-\lambda)^{-1}P\divv(\bar u_\Xi\otimes Z+Z\otimes \bar u_\Xi)-(\Dx-\lambda)^{-1}\alpha_L\sum_{\xi_j\in\Xi}(Z,\tilde \psi_{j,L})\varphi_{j,L}=\\
(\Dx-\lambda)^{-1}P\divv(\bar u_\Xi\otimes Z+Z\otimes \bar u_\Xi+(\Dx-\lambda)^{-1}(\sum_{\xi_j\in\Xi}(Z,\tilde\psi_j)(\bar u_j\otimes\varphi_j+\varphi_j\otimes\bar u_j)))\\+
(\Dx-\lambda)^{-1}(\Pi_{\Xi,L}Z-\Pi_\Xi Z):=\tilde Z+\bar Z,
\end{multline}
where $\Pi_\Xi=\sum_{\xi_j\in\Xi}\Pi_j$,
and, therefore, due to the structure of functions $\bar u_\Xi$, $\tilde\psi_{0,L}$ and $\tilde\varphi_{0,L}$, we have
$$
\|Z\|_{H^2_b}\le C\sup_{\xi_j\in\Xi}\|Z\|_{H^1_b(B^{\alpha L}_{\xi_j})}.
$$
Thus, without loss of generality, we may assume that
\begin{equation}\label{2.0}
\|Z\|_{H^1_b(B^{\alpha L}_0)}\ge \delta\|Z\|_{H^2_b}
\end{equation}
for some $\delta>0$ which is independent of $L$. Let $\Phi_L$ be the cut-off function, the same as in Section \ref{s0} and let $\tilde S$ be the stream function of $\tilde Z$. Then, analogously to \eqref{1.s0}, we see from \eqref{2.Z} that $\tilde S$ can be presented in the form of a Zygmund--Calderon operator applied to the function from $L^\infty$. Thus, $\tilde S\in BMO$ and therefore
\begin{multline}\label{2.tS}
|\tilde S(x)|\le C\|Z\|_{H^2_b}\ln(1+|x|)\le C_1\|Z\|_{H^1_b}\ln(1+|x|),\\ \|\nabla \tilde S\|_{H^1_b}\le C\|Z\|_{H^1_b}.
\end{multline}
{\color{black}
see \cite{FS72} (and also \cite{Pat23} for the elementary and
self-contained proof of the logarithmic estimate).}

On the other hand, using \eqref{1.cl} and arguing as in the proof of Lemma \ref{Lem2.pro}, we get
\begin{equation}\label{2.bZ}
\|\bar Z\|_{H^2_b}\le CL^{-1/2}\|Z\|_{H^1_b}.
\end{equation}
We define then $Z_L:=\nabla^\bot(\Phi_L\tilde S)=\Phi_L\tilde Z+\nabla^{\bot}\Phi_L\tilde S$. The second term here is of order $L^{-1}\ln L$ and we get
\begin{equation}\label{2.cut}
\|\Phi_LZ-Z_L\|_{H^2_b}\le CL^{-1/2}\|Z\|_{H^1_b}.
\end{equation}
We multiply equation \eqref{2.aue} on $\Phi_L$ and use that
$$
\|\Phi_L(\Dx-\lambda)Z -(\Dx-\lambda)(\Phi_LZ)\|_{L^2_b}\le CL^{-1}\|Z\|_{H^1_b}.
$$
Moreover,
\begin{multline}
\Phi_L P\divv(\bar u_\Xi\otimes Z+Z\otimes\bar u_\Xi)=\Phi_LP\div(\bar u_0\otimes (\Phi_LZ)+(\Phi_LZ)\otimes\bar u_0)+\\+\Phi_LP\divv(\sum_{\xi_j\ne0}\bar u_0\otimes Z+Z)\otimes\bar u_0)
\end{multline}
and using that the integral operator $P\divv$ has cubically decaying kernel, we infer that
$$
\|\Phi_L P\divv(\bar u_\Xi\otimes Z+Z\otimes\bar u_\Xi)-P\div(\bar u_0\otimes (\Phi_LZ)+(\Phi_LZ)\otimes\bar u_0)\|_{L^2_b}\!\le\! CL^{-3}\|Z\|_{H^1_b}.
$$
Analogously,
$$
\|\Phi_L\Pi_{\Xi,L}Z-\Pi_{0,L}(\Phi_L)Z\|_{L^2_b}\le CL^{-3}\|Z\|_{H^1_b}.
$$
Combining all of the obtained estimates together, we see that the function $Z_L$ solves
$$
\Cal L_0Z_L+\eb\Pi_{0,L}Z_L=g_L,
$$
where the function $g_L$ satisfies the estimate
$$
\|g_L\|_{L^2_b}\le CL^{-1/2}\|Z\|_{H^1_b}
$$
and Lemma \ref{Lem1.auL} gives us that
$$
\|Z_L\|_{H^2_b}\le CL^{-1/2}\|Z\|_{H^1_b}.
$$
Finally, using \eqref{2.cut} and the definition of $\Phi_L$, we arrive at
$$
\|Z\|_{H^2_b(B^{\alpha L}_0)}\le CL^{-1/2}\|Z\|_{H^1_b}
$$
which contradicts \eqref{2.0} if $L$ is large enough and finishes the proof of the lemma and the theorem as well.
\end{proof}
\begin{remark} If the function $g$ is more regular, i.e. $g\in H^s$ (resp. $H^s_{\theta_{x_0}}$, $H^sA^\infty_{\theta_{x_0}^{5/6}}$), then the solution $u$ belongs to $H^{s+2}$ (resp. $H^{s+2}_{\theta_{x_0}}$ or $H^{s+2}A^\infty_{\theta_{x_0}^{5/6}}$). for $s>0$. To get this result, it is sufficient to differentiate equation \eqref{2.au} and use bootstrapping arguments.
\end{remark}

\section{The auxiliary equation: the case of a bounded domain}\label{s3}
We now adapt the method to the case of bounded domains.
 To this end, we introduce one more big parameter $N$
 (which will be related to the number of vortices in what follows)
 and take $M:=cNL$ where $c$ is a fixed constant. We also consider a
 smooth bounded domain $\Omega$ such that $0\in\Omega$ (for simplicity,
 we assume that $\Omega$ is simply connected although
 this assumption is not important) and consider the scaled
 domain $\Omega_M:=M\Omega$. Finally, we put an extra condition
on the subset $\Xi\subset L\Bbb Z^2$, namely, that
\begin{equation}\label{3.bd}
d(\Xi,\partial\Omega_M)\ge NL=M/c,\ \ \Xi\in\Omega_M.
\end{equation}
In other words, we assume that the centers of all vortices are inside
 $\Omega_M$ and are situated at a distance of at least $M/c$
 from the boundary. Obviously, the number of such centers
 is finite $\#\Xi<\infty$ and, fixing  $c$, we may assume that the
 number of such vortices is proportional to $N^2$:
$$
\#\Xi\sim c_1 N^2.
$$
Then, we introduce a multivortex $\bar u_\Xi$ via \eqref{2.m} and
consider the linearization of the Navier--Stokes problem in
$\Omega_M$ with Dirichlet boundary conditions perturbed by the
approximative spectral projector:
\begin{equation}\label{3.lin}
\begin{cases}
(\Dx\!-\!\lambda)u\!-\!\div(\bar u_{\Xi}\otimes u\!+\!u\otimes\bar u_\Xi)\!+\!\Nx p\!+\!\eb \Pi_{\Xi,L}u\!=\!g:=\divv \tilde g\!+\!\hat g,\\ \divv u=0,\ u\big|_{\partial\Omega_M}=0,
\end{cases}
\end{equation}
where as before $\tilde g=\sum_{\xi_j\in\Xi}\tilde g_j$ and
$\supp \tilde g_j\subset B^1_{\xi_j}$. The main result of this
section is the following theorem.
\begin{theorem}\label{Th3.main} Let $L$, $N$ be large enough and
let $\eb>0$ be small enough. Then, for every $g\in L^2(\Omega_M)$,
problem \eqref{3.lin} possesses a unique solution $u\in H^2_\sigma(\Omega_M)$
and the following estimate holds:
\begin{equation}\label{3.main}
\|u\|_{H^2}\le C\|g\|_{L^2},
\end{equation}
where the constant $C$ depends on $\eb$, but is independent of $L$ and $N$ (recall that $M:=cNL$ and $c$ is fixed).
\end{theorem}
\begin{proof} The proof of this theorem is similar to the proof of
Theorem~\ref{Th2.main} (and is actually based on it), but we need to
add some extra estimates related with the influence of the boundary.
 Note that the number of vortices is now finite, so we need not
to care of the uniqueness, so only the existence of a solution
should be verified. The first step is exactly the same as in
Theorem~\ref{Th2.main}, namely, to get rid of function $\hat g$.
\begin{lemma}\label{Lem3.lin} For sufficiently small $\eb>0$ and
sufficiently large $L$ and $N$ the problem
\begin{equation}\label{3.sp}
(\Dx -\lambda)\hat u+\nabla \hat p+\eb\Pi_{\Xi,L}\hat u=\hat g,\ \
\divv\hat u=0,\ \ \hat u\big|_{\partial\Omega_M}=0
\end{equation}
has a unique solution $\hat u\in H^2_{\sigma}(\Omega_M)$ and the
following estimate holds
\begin{equation}\label{3.linest}
\|\hat u\|_{H^2(\Omega_M)}\le C\|\hat g\|_{L^2(\Omega_M)},
\end{equation}
where the constant $C$ is independent of $L$ and $N$.
\end{lemma}
\begin{proof} It is sufficient to verify the lemma for $\eb=0$ only
(since the operator $\Pi_{\Xi,L}$ is uniformly bounded, the result
for small $\eb$ will follow via the perturbation arguments).
The uniform estimate for the $H^1$-norm of the solution $\hat u$
follows from the basic  energy estimate (obtained by multiplication
of the equation by $\hat u$). After that the uniform estimate for the
$H^2$-norm follows
{\color{black}from the elliptic regularity of the Stokes operator
if we put the term $-\lambda\hat u$ to the right-hand side.}
\end{proof}
Thus, the function $v:=u-\hat u$ solves equation \eqref{3.lin} with $\hat g=0$ and $\tilde g$ replaced by
$$
\tilde g+\tilde R,\ \ \tilde R:=\bar u_\Xi\otimes \hat u+\hat u\otimes\bar u_\Xi
$$
and the remainder $\tilde R$ satisfies
$$
\|\tilde R\|_{H^1}\le C\|\hat g\|_{L^2},
$$
which, in turn, gives us the desired reduction to the case $\hat g=0$.
\par
At the next step, we solve equation \eqref{3.lin} for $v$ in
the whole plane using Theorem \ref{Th2.main}. Let $\tilde u$ be
a solution of this problem. Then, $w:=u-\hat u-\tilde u$ solves
the non-homogeneous analogue of \eqref{3.lin}
\begin{equation}\label{3.linn}
\begin{cases}
(\Dx\!-\!\lambda)w\!-\!\div(\bar u_{\Xi}\otimes w\!+\!w\otimes\bar u_\Xi)\!+\!\Nx p\!+\!\eb \Pi_{\Xi,L}w\!=\!0\!,\\ \divv w=0,\ w\big|_{\partial\Omega_M}=\tilde u\big|_{\partial\Omega_M}.
\end{cases}
\end{equation}
To complete the proof of the theorem, we need to construct an approximate
solution of problem \eqref{3.linn}. To this end, we need to check that
$\tilde u\big|_{\partial M}$ is small enough and to construct a divergent
free extension of it inside  the domain without losing the smallness
property. This will be done in the next lemmata.
\begin{lemma}\label{Lem3.small} For sufficiently large $N$ and $L$, we have
\begin{equation}
\|\tilde u\big|_{\partial\Omega_M}\|_{H^{3/2}(\partial\Omega_M)}\le CM^{-1}\|g\|_{L^2},
\end{equation}
where the constant $C$ is independent of $M$ and $L$.
\end{lemma}
\begin{proof} According to the weighted estimate of Theorem \ref{Th2.main},
$$
\|\tilde u\|_{H^2(B^1_{x_0})}^2\le C\|\tilde u\|^2_{H^2_{\theta_{x_0}}}\le
C\int_{y\in\R^2}\|\tilde g+\tilde R\|^2_{H^1(B^1_y)}\theta(y-x_0)\,dy.
$$
Using the fact that the supports of functions $\tilde g$ and $\tilde R$ are
at a distance  of at least $M/c$ from the boundary, after the
integration over $x_0\in \Cal O_1(\partial\Omega_M)$
(which is the $1$-neighbourhood of the boundary)
{\color{black}and using \eqref{0.eq.1}}, we get
$$
\|\tilde u\|^2_{H^2(\Cal O_1(\partial\Omega_M))}\le CM^{-3}|\Cal O_1(\partial\Omega_M)|\|\tilde g+\tilde R\|_{H^1}^2\le CM^{-2}\|g\|^2_{L^2(\Omega_M)}
$$
and the desired estimate follows from the trace theorem. This finishes the proof of the lemma.
\end{proof}
\begin{lemma}\label{Lem3.ext} There exists a divergence free extension $u_b\in H^2_\sigma(\Omega_M)$ which satisfies
\begin{equation}\label{3.ext}
u_b\big|_{\partial\Omega_M}=-\tilde u\big|_{\partial\Omega_M},\ \
\|u_b\|_{H^2(\Omega_M)}\le CM^{-1/2}\|g\|_{L^2}.
\end{equation}
where the constant $C$ is independent of $M$ and $L$.
\end{lemma}
\begin{proof} We use the standard extension procedure
(see, for instance, \cite[Proposition~I.2.4]{Temam1},
 presenting the desired function $u_b$ in the form
\begin{equation}\label{ub}
u_b=\nabla U_b+\nabla^\bot V_b,
\end{equation}
where $U_b$ is a harmonic function, which solves
$$
\Dx U_b=0,\ \ \partial_n U_b\big|_{\partial\Omega_M}=-\tilde u\big|_{\partial\Omega_M}\cdot n.
$$
Since $\tilde u$ is divergence free, the solvability condition for
$U_b$ is satisfied and we only need to control the dependence
of the norms of $U_b$ on $M$. To get such estimates, we introduce
the scaled function $\hat U_b(x):=U_b(M x)$ which is harmonic in $\Omega$.
 Then, using the fact that the homogeneous $H^s$-norms of $U_b$ in the
domain are scaled as the  homogeneous $H^{s-3/2}$ norms of $\partial_n U_b$
on the boundary for $s=2,3$, together with the elliptic regularity in $\Omega$,
 we end up with
\begin{equation}\label{3.hom}
\|U_b\|_{\dot H^2(\Omega_M)}+\|U_b\|_{\dot H^3(\Omega_M)}\le C\|\tilde u\|_{H^{3/2}(\partial\Omega_M)}\le C_1M^{-1}\|g\|_{L^2}.
\end{equation}
For $s=1$, since we do not have the uniform control of $H^{-1/2}$-norm
on the boundary,
we use the non scale invariant  estimate (following from the elliptic regularity)
$$
\|\hat U_b\|_{\dot H^1(\Omega)}\le C\|\partial_n \hat U_b\|_{L^2(\partial\Omega)},
$$
which gives us
$$
\|U_b\|_{\dot H^1(\Omega_M)}\le CM^{1/2}\|\partial_n U_b\|_{L^2(\Omega_M)}.
$$
Thus, we get
\begin{equation}\label{3.Ub}
\|\nabla U_b\|_{H^2(\Omega_M)}\le CM^{-1/2}\|g\|_{L^2(\Omega_M)}.
\end{equation}
It only remains to construct the function $V_b$.
We see from~\eqref{ub} that
$$
u_b\cdot n=\partial_nU_b+\partial_\tau V_b,
$$
so that
$$
\partial_\tau V_b=u_b\cdot n-\partial_nU_b=
-\tilde u\big|_{\partial\Omega_M}\cdot n-\partial_nU_b=0,
$$
and without loss of generality we may assume that $V_b\big|_{\partial\Omega_M}=0$.

Next, from \eqref{ub} and \eqref{3.Ub} we obtain that on $\partial\Omega_M$
$$
\partial_nV_b=-\tilde u\cdot\tau-\partial_\tau U_b\in
H^{3/2}(\partial\Omega_M),\quad
\|\partial_n V_b\|_{H^{3/2}(\partial\Omega_M)}\le CM^{-1/2}\|g\|_{L^2}.
$$
The construction of $V_b$ is at our disposal, therefore we may set
$\partial^2_nV_b=0$ on $\partial\Omega_M$ and apply the extension
theorem for Sobolev spaces, which says that there exists a scale
invariant constant $K$ and an extension of $V_b$ from the boundary to
$\Omega_M$ (which we still denote by $V_b$) such that
\begin{multline*}
\|V_b\|_{H^3(\Omega_M)}\le K\left(\|V_b\|_{H^{5/2}(\partial\Omega_M)}+
\|\partial_nV_b\|_{H^{3/2}(\partial\Omega_M)}+
\|\partial^2_nV_b\|_{H^{1/2}(\partial\Omega_M)}\right)\\
=K
\|\partial_nV_b\|_{H^{3/2}(\partial\Omega_M)}
\le CM^{-1/2}\|g\|_{L^2},
\end{multline*}
where the constant $C$ is independent of $M$. This finishes the proof of the lemma.

\end{proof}

For future references it is convenient to
single out the key estimate in the previous lemma as the following
corollary.

\begin{corollary}\label{Cor:M1/2}
Let $v=\nabla_xQ$, $\Delta_x Q=0$ in $\Omega_M$ and
let $v\cdot n=\partial_{ n}Q=\eta$ on $\partial\Omega_M$,
where $\int_{\partial\Omega_M}\eta(l)dl=0$.
Then
\begin{equation}\label{M1/2}
\|v\|_{H^2(\Omega_M)}\le CM^{1/2}\|\eta\|_{H^{3/2}(\partial\Omega_M)}.
\end{equation}
\end{corollary}


We are now ready to finish the proof of the theorem. We fix
$$
u_{app}:=\hat u+\tilde u+u_b.
$$
Then, according to the obtained estimates
$$
\|u_{app}\|_{H^2(\Omega_M)}\le
C\(\|\tilde g\|_{H^1(\Omega_M)}+\|\hat g\|_{L^2(\Omega_M)}\)
$$
and the remainder $R=\divv \tilde R+\hat R$ satisfies
$$
\|\tilde R\|_{H^1(\Omega_M)}+\|\hat R\|_{L^2(\Omega_M)}\le
CM^{-1/2}\(\|\tilde g\|_{H^1(\Omega_M)}+\|\hat g\|_{L^2(\Omega_M)}\).
$$
Iterating this procedure, we get the exact solution for the equation
\eqref{3.lin} and finish the proof of the theorem.
\end{proof}
\begin{remark}\label{Rem3.bad} Unfortunately, in contrast to the case
of $\Omega=\R^2$, we do not have here the weighted and uniformly
local analogue of estimate \eqref{3.main}. The key problem here is
that we do not know how to verify the weighted estimate for the Stokes problem
$$
(\Dx -\lambda)u+\Nx p=g,\ \ \divv u=0,\ \ u\big|_{\partial\Omega_M}=0.
$$
We believe that this estimate remains true and return to this problem somewhere else.
\end{remark}

\section{The Lyapunov--Schmidt reduction}\label{s4}
The aim of this section is to get rid of the auxiliary spectral
projector $\Pi_{\Xi,L}$ and to find the invariant spectral subspaces
for the operator $\Cal L_\Xi$. As before, we will consider two cases,
namely, when $\Omega=\R^2$ and $\Omega=\Omega_M$. In the second case,
we naturally replace the operator $\Cal L_\Xi$ by
$$
\Cal L_{\Xi,\Omega_M}v:=P_{\Omega_M}(\Dx-\lambda)v-
P_{\Omega_M}\divv(\bar u_\Xi\otimes v+v\otimes\bar u_\Xi),\ \
$$
where $P_{\Omega_M}: L^2(\Omega_M)\to L^2_\sigma(\Omega_M)$ is the
Leray projector to the divergence free vector fields in $L^2(\Omega_M)$.
To be more precise,
$$
\Cal L_{\Xi,\Omega_M}: D(\Cal L_{\Xi,\Omega_M}):=
L^2_\sigma(\Omega_M)\cap H^1_0(\Omega_M)\cap H^2(\Omega_M)\to L^2_\sigma(\Omega_M).
$$
Note that, by the construction $\Pi_{\Xi,L}$ is the projector not
only in $L^2_\sigma(\R^2)$ but also in $L^2_\sigma(\Omega_M)$, so we may
split $L^2_\sigma$ in a direct sum
$$
L^2_\sigma(\Omega)=(1-\Pi_{\Xi,L})L^2_\sigma(\Omega)\oplus\Pi_{\Xi,L}L^2_\sigma(\Omega):=
H_{hyp}\oplus H_{neu},
$$
where $\Omega=\R^2$ or $\Omega=\Omega_M$ and write out the operator
$\Cal L_\Xi$ as a matrix
\begin{equation}\label{4.split}
\Cal L_\Xi\!=\!\!\(\begin{matrix}(1-\Pi_{\Xi,L})\Cal L_\Xi(1-\Pi_{\Xi,L}),
& (1-\Pi_{\Xi,L})\Cal L_\Xi\Pi_{\Xi,L}\\ \Pi_{\Xi,L}\Cal L_\Xi(1-\Pi_{\Xi,L}),
& \Pi_{\Xi,L}\Cal L_\Xi\Pi_{\Xi,L}\end{matrix}\)\!:=\!\!\(\begin{matrix} \Cal L_{11},
&\!\Cal L_{12}\\\Cal L_{21},&\!\Cal L_{22}\end{matrix}\)
\end{equation}
and the analogous  presentation works for $\Cal L_{\Xi,\Omega_M}$.
Note that the spaces $H_{neu}$ and $H_{hyp}$ are not invariant with
respect to operator $\Cal L_\Xi$ (resp $\Cal L_{\Xi,\Omega_M}$), but
as it will be shown below, they are almost invariant if $L$ is large enough.
Therefore, since the exponential dichotomy is preserved
under small perturbations, we derive from this fact the
existence of truly invariant spectral subspaces isomorphic
to $H_{hyp}$ and $H_{neu}$. The next theorem can be considered
as the main result of this section.
\begin{theorem}\label{Th4.main} Let $L$ be large enough. Then,
there exist bounded  (in $L^2_\sigma(\R^2)$ and in $L^2_{b,\sigma}(\R^2)$)
operators
$$
K_{hyp}: H_{hyp}\to H_{neu},\ \ K_{neu}: H_{neu}\to H_{hyp}
$$
such that the spaces
\begin{equation}\label{4.inv}
M_{hyp}:=\{u+K_{hyp}u,\ u\in H_{hyp}\},\ \ M_{neu}:=\{u+K_{neu}u,\ \ u\in H_{neu}\}
\end{equation}
are invariant with respect to $\Cal L_\Xi$. Moreover,
\begin{equation}\label{4.spec}
\sigma\(\Cal L_{\Xi}\big|_{M_{hyp}}\)\cap B^\delta_0=\varnothing,\ \
\sigma\(\Cal L_{\Xi}\big|_{M_{neu}}\)\subset B^{CL^{-1/2}}_0,
\end{equation}
where the positive constants $C$ and $\delta$ are independent of $L$.
\par
The analogous result also holds for the operator
$\Cal L_{\Xi,\Omega_M}$ in  $L^2_\sigma(\Omega_M)$.
\end{theorem}
\begin{proof} Note that the operator $\Cal L_\Xi$ is
unbounded, so the invariance of the space $M_{hyp}$ should be understood as
$$
\Cal L_\Xi: D(\Cal L_\Xi)\cap M_{hyp}\to M_{hyp}.
$$
As we will see later, the operator $\Cal L_\Xi\big|_{M_{neu}}$
is bounded, so the invariance may be understood in a usual sense.
The next lemma, which establishes the smallness of the operators
$\Cal L_{12},\Cal L_{22}$ and $\Cal L_{21}$ is one of the key tools
for proving the theorem.
\begin{lemma}\label{Lem4.small} Let $L$ be large enough. Then,
\begin{equation}\label{4.small}
\|\Pi_{\Xi,L}\Cal L_\Xi\|_{\Cal L(H^2,H^2)}+
\|\Cal L_\Xi\Pi_{\Xi,L}\|_{\Cal L(L^2,H^2)}\le CL^{-1/2},
\end{equation}
where the constanct $C$ is independent of $L$ and the same
result holds for the $L^2_b$-spaces.
\par
Moreover, the analogous estimates hold for the operator $\Cal L_{\Xi,\Omega_M}$
in the space $L^2_\sigma(\Omega_M)$.
\end{lemma}
\begin{proof} We give the proof for the most complicated case of the
operator $\Cal L_{\Xi,\Omega_M}$. The case $\Omega=\R^2$ is analogous,
but simpler since we need not to control the difference $P-P_{\Omega_M}$
and the result follows in a straightforward way from from Lemma \ref{Lem1.cut}.
Let us start with the operator $\Cal L_{\Xi,\Omega_M}\Pi_{\Xi,L}$.
By definition
\begin{equation}\label{4.term}
\Cal L_{\Xi,\Omega_M}\Pi_{\Xi,L}u=\alpha_L\sum_{ j}(u,\psi_{j,L})\Cal L_{\Xi,\Omega_M}\varphi_{j,L}=\alpha_L\sum_{j}(u,\tilde\psi_{j,L})\Cal L_{j,\Omega_M}\varphi_{j,L},
\end{equation}
since the supports of $\bar u_i$ and $\varphi_{j,L}$ do not intersect if $i\ne j$. Here and below
$$
\Cal L_{j,\Omega_M}u:=P_{\Omega_M}(\Dx-\lambda)u-P_{\Omega_M}\divv(\bar u_j\otimes u+u\otimes\bar u_j)
$$
and the operator $\Cal L_j$ is defined analogously replacing the projector $P_{\Omega_M}$ by $P$ and extending the function $u$ by zero for $x\notin\Omega_M$.
\par
At the next step, we get rid of $P_{\Omega_M}$. To this end, we note that
$$
(\Cal L_{j,\Omega_M}-\Cal L_j)\varphi_{j,L}=
(P_{\Omega_M}-P)((\Dx-\lambda)\varphi_{j,L}-\divv(\bar u_j\otimes\varphi_{j,L}+\varphi_{j,L}\otimes\bar u_j))=\Nx Q,
$$
where $Q$ is a harmonic function satisfying
\begin{multline*}
\Dx Q=0,\ \partial_n Q\big|_{\partial\Omega_M}=\\=-P((\Dx-\lambda)\varphi_{j,L}-\divv(\bar u_j\otimes\varphi_{j,L}+\varphi_{j,L}\otimes\bar u_j))\big|_{\partial\Omega_M}\!\cdot n=:h_M\big|_{\partial\Omega_M}\cdot n.
\end{multline*}
Since the kernel of $\Nx\circ P$ and all its $x$-derivatives decay cubically,
 $P\varphi_{j,L}=\varphi_{j,L}$  and the distance from any $\xi_j\in\Xi$
  to the boundary is at least $M/c$, we have for each
  Sobolev norm
$$
\|h_M\|^2_{H^k(\mathcal O_1(\partial\Omega_M))}\le C_k|\mathcal O_1(\partial\Omega_M)|
\frac1{M^6}\le C\frac 1{M^5}.
$$
For $k=2$ the trace theorem gives that
$$
\|h_M\|_{H^{3/2}(\partial\Omega_M)}\le CM^{-5/2}
$$
and Corollary~\ref{Cor:M1/2} gives that
\begin{equation}\label{4.Q}
\|\Nx Q\|_{H^2(\Omega_M)}\le CM^{-2}.
\end{equation}


Therefore, using the fact that $\#\Xi\sim N^2\sim M^2/L^2$, we have
$$
\|\sum_j(u,\tilde\psi_{j,L})(\Cal L_j-\Cal L_{j,\Omega_M})
\varphi_{j,L}\|_{H^2(\Omega_M)}\le CL^{-2}\|u\|_{L^2(\Omega_M)}
$$
and we only need to estimate the modified right-hand side of
\eqref{4.term}, where $\Cal L_{j,\Omega_M}$ is replaced by $\Cal L_j$.
To this end, we use estimate \eqref{1.ap-loc} with $s=4$ which gives us
the point-wise estimates
$$
|(\Cal L_j\varphi_{j,L})(x)|+|(\Nx\Cal L_j\varphi_{j,L})(x)|+
|(\Nx\Nx\Cal L_j\varphi_{j,L})(x)|\le CL^{-1/2}\theta(x-\xi_j)^{5/6}.
$$
Thus, we have to estimate the following quantity
$$
I:=\|\sum_j(u,\tilde\psi_{j,L})\Cal L_j\varphi_{j,L}\|_{H^2}^2= \sum_{i,j}(u,\tilde\psi_{j,L})(u,\tilde\psi_{i,L})  (\Cal L_j\varphi_{j,L},\Cal L_{i}\varphi_{i,L})_{H^2}.
$$
Using the above point-wise estimates, we see that
$$
|(\Cal L_j\varphi_{j,L},\Cal L_{i}\varphi_{i,L})_{H^2}|\le CL^{-1}(\theta^{5/6}_{\xi_i},\theta_{\xi_j}^{5/6})\le C_1L^{-1}\theta^{5/6}(\xi_i-\xi_j).
$$
Finally, using that the functions $\tilde \psi_{j,L}$ are orthogonal, the Bessel inequality gives
\begin{multline}
I\le CL^{-1}\sum_{i.j}((u,\tilde\psi_{i,L})^2+(u,\tilde\psi_{j,L})^2)\theta^{5/6}(\xi_i-\xi_j)\le \\\le C_1L^{-1}\sum_j(u,\tilde\psi_{j,L})^2\le C_2L^{-1}\|u\|^2_{L^2}
\end{multline}
and the desired estimate for $\Cal L_{\Xi,\Omega_M}\Pi_{\Xi,L}$ is proved.

Let us now consider the operator $\Pi_{\Xi,L}\Cal L_{\Xi,\Omega_M}$. Indeed,
using the {\color{black}`differentiated' estimate \eqref{1.uni}
in the form $|\nabla_x\nabla_x\varphi_{0,L}(x)|\le C\theta(x)$}
 and arguing as before, we have
\begin{equation}\label{est}
\|\Pi_{\Xi,L}\Cal L_{\Xi,\Omega_M}u\|_{H^2(\Omega_M)}^2\le C\sum_j(\Cal L_{\Xi,\Omega_M}u,\tilde \psi_{j,L})^2,
\end{equation}
so, we only need to estimate the scalar products in the right-hand side. To this end, we write
\begin{multline}
(\Cal L_{\Xi,\Omega_M}u,\tilde\psi_{j,L})\!\!=\!\!(u,\Cal L_{j,\Omega_M}^*(P_{\Omega_M}\tilde\psi_{j,L}))\!-\!\!\sum_{i\ne j}(\bar u_i\otimes u+u\otimes\bar u_i,\!\Nx P_{\Omega_M}\tilde\psi_{j,L})+\\+(\partial_nu,
P_{\Omega_M}\tilde \psi_{j,L})_{\partial\Omega_M}=:J_{1,j}+J_{2,j}+J_{3,j}.
\end{multline}
At the next step, we replace $P_{\Omega_M}\tilde \psi_{j,L}$ by $P\tilde \psi_{j,L}$.  We start with the term $J_{3,j}$. Since the kernel of $P$ has a quadratic decay rate, then, arguing as before, we have
$$
|P\tilde \psi_{j,L}(x)|_{\partial\Omega_M}\le CM^{-2}
$$
and
$$
|(\partial_n u, P\tilde\psi_{j,L})|_{\partial\Omega_M}\le C\|u\|_{H^2}\|P\tilde\psi_{j,L}\|_{L^2(\partial\Omega_M)}\le CM^{-3/2}\|u\|_{H^2}.
$$
We now  use the first estimate in \eqref{4.gradP} from the next Proposition~\ref{Prop4.M},
which along with the previous estimate gives that
$$
|(\partial_n u, P_M\tilde\psi_{j,L})|_{\partial\Omega_M}\le 
 CM^{-1}\|u\|_{H^2},
$$
and, therefore,
$$
\sum_{j}J_{3,j}^2\le CM^{-2}\#\Xi\|u\|^2_{H^2}\le C_1L^{-2}\|u\|^2_{H^2}.
$$
For estimating the terms $J_{2,j}$ and $J_{1,j}$, we need the following proposition.

\begin{proposition}\label{Prop4.M} The following estimate holds
\begin{equation}\label{4.gradP}
\aligned
&\| P\tilde\psi_{j,L}-P_{\Omega_M}\tilde\psi_{j,L}\|_{H^2(\Omega_M)}\le C M^{-1},\\
&\| P\Nx\tilde\psi_{j,L}-\Nx P_{\Omega_M}\tilde\psi_{j,L}\|_{H^2(\Omega_M)}\le C M^{-2},
\endaligned
\end{equation}
where $C$ is independent of $N$ and $L$.
\end{proposition}

\begin{proof} Indeed, since these differences are  harmonic functions,
we only need to estimate the trace of these functions on the boundary
$\partial\Omega$.
Staring with the second and more difficult  case, this difference is
equal to $\Nx Q$, where $Q$ solves
$$
\Dx Q=0,\ \ \partial_nQ\big|_{\partial\Omega_M}=(P\Nx\tilde\psi_{j,L}-
\Nx P_{\Omega_M}\tilde\psi_{j,L})\cdot  n\big|_{\partial\Omega_M}.
$$
 Since the operator $\Nx\circ P$ has a cubically decaying kernel, arguing
 as before, we have the following estimate
$$
\|\Nx P\tilde\psi_{j,L}\|_{H^{3/2}(\partial\Omega_M)}\le CM^{-5/2}.
$$
The second term is a bit more delicate. We use that
$P_{\Omega_M}=-\Nx^\bot(\Dx)_D^{-1}\Rot$, where $(\Dx)_D$ is the
Laplace operator with homogeneous Dirichlet boundary conditions.
To get the desired estimate, we need to estimate the derivatives for
the Green function $G_{\Omega_M}(x,y)$ of this operator, namely, we need that
\begin{equation}\label{4.ker}
|D_x^\alpha D_y^\beta G_{\Omega_M}(x,y)|\le
C_{\alpha,\beta}\frac1{|x-y|^{|\alpha|+|\beta|}}
\end{equation}
uniformly with respect to $M$ and with  $|\alpha|+|\beta|\ge1$.
Assume that this estimate holds. Then, for $y\in B_{\xi_j}^L$
and $x\in O_1(\partial\Omega_M)$, we have
$$
|\Nx P_{\Omega_M}\tilde \psi_{j,L}(x)|=\biggl|\int_{\Omega_M}\Nx\Nx^\bot
\Rot_y G_{\Omega_M}(x,y)\tilde \psi_{j,L}(y)\,dy\biggr| \le CM^{-3}
$$
and the same is true for the derivatives. Therefore
$$
\|\Nx P_{\Omega_M}\tilde\psi_{j,L}\|_{H^{3/2}(\partial\Omega_M)}\le CM^{-5/2}.
$$
Assuming \eqref{4.ker} this
 gives the desired estimate by Corollary~\ref{Cor:M1/2}.

 Although  estimate \eqref{4.ker} is standard, for completeness, we briefly indicate a possible way to prove it.
 First of all, we may do scaling and reduce the problem to the fixed domain $\Omega$. Since $\Omega$ is simply connected and smooth there exist a conformal map $w:\Omega\to B^1_0$ which is smooth up to the boundary and, due to conformal invariance of Green functions, we have
 $$
 G_\Omega(x,y)=G_{B^1_0}(w(x),w(y)).
 $$
 {\color{black} see e.g. \cite{Con}.}
 Since $G_{B^1_0}$ can be written explicitly and $|w(x)-w(y)|\sim |x-y|$, the straightforward calculations give the desired estimate \eqref{4.ker}.

 The first estimate in \eqref{4.gradP} follows from \eqref{P-P},
Corollary~\ref{Cor:M1/2} and the estimate
$$
\|P\tilde\psi_{j,L}\|_{H^{3/2}(\partial\Omega_M)}\le CM^{-3/2},
$$
and this completes the proof of  Proposition \ref{Prop4.M}.
\end{proof}
 Let us now estimate $J_{2,j}$. We want to replace $P_{\Omega_M}$ by $P$ in each
 term of it (we denote by $\tilde J_{2,j}$ the modified terms).
 We write
\begin{multline*}
|J_{2,j}^2-\tilde J_{2,j}^2|=\bigl|\left((\bar u_\Xi-\bar u_j)\otimes u
+ u\otimes(\bar u_\Xi-\bar u_j),
\nabla_x P_{\Omega_M}\tilde\psi_{j,L}-\nabla_xP\tilde\psi_{j,L}\right)\cdot\\\cdot
\left((\bar u_\Xi-\bar u_j)\otimes u
+ u\otimes(\bar u_\Xi-\bar u_j),
\nabla_x P_{\Omega_M}\tilde\psi_{j,L}+\nabla_xP\tilde\psi_{j,L}\right)\bigr|
\!\le\! C\|u\|_{L^2}^2M^{-2},
\end{multline*}
where we used
 the estimate from Proposition \ref{Prop4.M},
the uniform (with respect to $j$) boundedness
of the $C$-norm of $\bar u_\Xi-\bar u_j$, and the
fact that $\nabla_xP\tilde\psi_{j,L}$ is uniformly
bounded in $L^2$ (and so is  $\nabla_xP_{\Omega_M}\tilde\psi_{j,L}$ for large $M$
thanks to the estimate from Proposition \ref{Prop4.M}).

Since  the number of vortices $\#\Xi\sim N^2$, we get
 $$
 \sum_j|(J_{2,j}^2-\tilde J_{2,j}^2)|\le C\|u\|^2_{L^2}M^{-2}\#\Xi\le C_1L^{-2}.
 $$

We now use that the kernel of $\Nx P$ has
cubic decay rate {\color{black}and take into account that $\bar u_i$ has
finite support of size $1$ around $\xi_i$, and
$|\bar u_i(x)|\le C\theta_{\xi_i}(x)^2$.}
Therefore,
setting $A_i:=(|u|^2,\theta_{\xi_i})^{1/2}$ we obtain
\begin{multline*}
|(u\otimes\bar u_i+\bar u_i\otimes u,\Nx P\tilde \psi_{j,L})|\le
C(|u|,\theta_{\xi_i}^2\theta_{\xi_j})\le
C(|u|^2,\theta_{\xi_i}^2)^{1/2}(\theta_{\xi_i}^2,\theta_{\xi_j}^2)^{1/2}\\\le
C(|u|^2,\theta_{\xi_i})^{1/2}(\theta_{\xi_i}^2,\theta_{\xi_j}^2)^{1/2}
\le C(|u|^2,\theta_{\xi_i})^{1/2}\theta(\xi_i-\xi_j)=CA_i\theta(\xi_i-\xi_j).
\end{multline*}
Thus,
\begin{multline}
\sum_{j}\tilde J_{2,j}^2\le C\sum_j\!\(\!\sum_{i\ne j}A_i\theta(\xi_i-\xi_j)\!\)^2\!\!\!
=C\sum_{i,k}\sum_{j\ne \{i,k\}}\!\!A_iA_k\theta(\xi_i-\xi_j)\theta(\xi_k-\xi_j)\le\\
\le C_1\sum_{i,k}\sum_{j\ne \{i,k\}}A_iA_k\theta(\xi_i-\xi_k)
(\theta(\xi_i-\xi_j)+\theta(\xi_k-\xi_j))\le\\
\le CL^{-3}\sum_{i,k}A_iA_k\theta(\xi_i-\xi_k)\le C_1L^{-3}\sum_i A_i^2=
\\=C_1L^{-3}(|u|^2,\sum_i\theta_{\xi_i})\le C_2L^{-3}\|u\|_{L^2}^2,
\end{multline}
and this gives the desired estimate for $\sum_jJ_{2,j}^2$.

Let us finally consider the terms~$J_{1,j}$. Here, we also want to
replace $P_{\Omega_M}\tilde\psi_{j,L}$ by
$\psi_{j,L}:=P\tilde\psi_{j,L}$, but here we need to act a bit more
accurately since we cannot do this in the term
$\lambda(u,P_{\Omega_M}\tilde\psi_{j,L})$ (if we do this directly,
the error will be of order $M^{-1}$ and this is not enough since
the number of terms in the sum is proportional to $N^2$). However,
we can do this in the terms
$(u,(L_{j,\Omega_M}^*+\lambda)\tilde\psi_{j,L})$ using Proposition \ref{Prop4.M}, since in these terms we have the combination $\Nx
P_{\Omega_M}$, so the error will be of order $L^{-2}$.
\par
After that we replace
$\Cal L_{j,\Omega_M}^*\psi_{j,L}$ by $\Cal L_j^*\psi_{j,L}$,
using the fact that $u\in L^2_\sigma(\Omega_M)$ and, therefore,
$$
(u,v)=(u,Pv)=(u,P_{\Omega_M}v).
$$
We also write
\begin{multline}
(u,\Cal L_j^*(P\tilde\psi_{j,L}))=(u,(\Cal L_j^*+\lambda)(P\tilde\psi_{j,L}))-\lambda(u,\tilde \psi_{j,L})=\\=(u,(\Cal L_j^*+\lambda)\psi_j)-\lambda(u,\psi_j)-\lambda(u,\tilde\psi_{j,L}-\tilde \psi_j)+
(u,(\Cal L_j^*+\lambda)P(\tilde\psi_{j,L}-P\tilde\psi_j))=\\=-\lambda(u,\tilde\psi_{j,L}-\tilde \psi_j)+
(u,(\Cal L_j^*+\lambda)P(\tilde\psi_{j,L}-P\tilde\psi_j)).
\end{multline}
Here we  used that $\psi_j=P\tilde\psi_j$ satisfies $\Cal
L_j^*\psi_j=0$. So, we only need to estimate the error terms. We
now recall that according to \eqref{1.cl}, {\color{black} we can freely assume that}
$$
|\tilde\psi_{j,L}(x)-\tilde\psi_j(x)|\le CL^{-2}\theta(x-\xi_j)
$$
and the analogous estimate holds for all derivatives
(note that the left-hand vanishes for $|x|\le\alpha L$). Note that
$$
(u,(\Cal L_j^*+\lambda)v)=(u,\Dx v+2(\bar u_j,\Nx) v+\bar u_j^\bot\Rot v).
$$
In particular, if we put $v=P(\tilde\psi_{j,L}-\tilde\psi_j)$, the
projector $P$ will appear always together with differentiation.
Since $\Nx \circ P$ has cubically decaying kernel, we may write
$$
|(u,(\Cal L_j^*+\lambda)P(\tilde\psi_{j,L}-\tilde\psi_j))|\le
 CL^{-2}(|u|,\theta_{\xi_j})
$$
and
\begin{multline}
\sum_j (|u|,\theta_{\xi_j})^2=
\sum_{j}\int_{x,y}|u(x)||u(y)|\theta(x-\xi_j)\theta(y-\xi_j)\,dx\,dy
\\\le C\int_{x,y}|u(x)||u(y)|\theta(x-y)\,dx\,dy\le C_1\|u\|^2_{L^2}.
\end{multline}
This finishes the proof of the desired estimate for $\Pi_{\Xi,L}\Cal L_{\Xi,\Omega}$.

 Thus, the lemma is proved for $\Omega=\Omega_M$. Note that
 the proof essentially used the fact that the number of vortices
 is finite and proportional to $N^2$, but this fact is used for
 replacing $P_{\Omega_M}$ by $P$ only. Since this is not necessary
 for $\Omega=\R^2$, we may take arbitrarily many vortices in this case.
\end{proof}
\begin{corollary}\label{Cor4.small} For sufficiently large $L$ and $N$, we have
\begin{equation}
\|\Cal L_{12}\|_{\Cal L(L^2,H^2)}+\|\Cal L_{22}\|_{\Cal L(L^2,H^2)}+
\|\Cal L_{21}\|_{\Cal L(H^2,H^2)}\le CL^{-1/2},
\end{equation}
where the constant $C$ is independent of $L$ and $N$.
\end{corollary}
Indeed, this is an immediate corollary of the previous lemma and
the fact that $\Pi_{\Xi,L}$ is uniformly bounded in $H^2(\Omega_M)$.
\par
The next lemma gives the uniform invertibility of the operator
$\Cal L_{11}=(1-\Pi_{\Xi,L})\Cal L_{\Xi,\Omega_M}(1-\Pi_{\Xi,L})$.
\begin{lemma}\label{Lem4.inv} Let $L$ and $N$ be large enough. Then
the operator $\Cal L_{11}$ is invertible and the following estimate holds:
\begin{equation}
\|\Cal L_{11}^{-1}\|_{\Cal L(L^2_\sigma, H^2_\sigma)}\le C,
\end{equation}
where the constant $C$ is independent of $L$ and $N$.
\end{lemma}
\begin{proof} Recall that the operator $\Cal L_{\Xi,\Omega_M}+\eb\Pi_{\Xi,L}$
is uniformly invertible if $L$ is large enough. If $\Pi_{\Xi,L}u=0$ or,
which is the same as $(u,\tilde\psi_{j,L})=0$ for all~$j$, then
$$
\Cal L_{\Xi,\Omega_M}u+\eb \Pi_{\Xi,L}u=\Cal L_{\Xi,\Omega_M}u=:g
$$
and, for such functions $g$, the solution $u=u_g$ can be restored by
 $$
 u_g:=(\Cal L_{\Xi,\Omega_M}+\eb\Pi_{\Xi,L})^{-1}g,
 $$
 which solves the equation
$$
\Cal L_{\Xi,\Omega_M}u_g=g.
$$
In order to describe  the image of operator $\Cal L_{\Xi,\Omega_M}$,
we need to find the conditions which guarantee that $u_g\in H_{hyp}$.
To this end, we multiply equation for $u_g$ by $\tilde\psi_{j,L}$. This gives
\begin{equation}\label{4.eq}
\eb\alpha_L(u_g, \tilde\psi_{j,L})= (g,\tilde\psi_{j,L})-
(\Cal L_{\Xi,\Omega_M}(\Cal L_{\Xi,\Omega_M}+\eb\Pi_{\Xi,L})^{-1}g,\tilde\psi_{j,L}),
\end{equation}
for all $j\in\Xi$, and we see that
$$
u_g\in H_{hyp}\quad\Leftrightarrow\quad
(g,\tilde\psi_{j,L})=
(\Cal L_{\Xi,\Omega_M}(\Cal L_{\Xi,\Omega_M}+\eb\Pi_{\Xi,L})^{-1}g,\tilde\psi_{j,L}).
$$

Now let $I$ be the image of operator $\Cal L_{\Xi,\Omega_M}$:
 $$I:=\Cal L_{\Xi,\Omega_M}(H_{hyp}\cap D(\Cal L_{\Xi,\Omega_M})).$$
 Then, according to \eqref{4.eq},
\begin{equation}\label{4.I}
\!\!I\!=\!\bigg\{g\in L^2_\sigma(\Omega_M)\!:\,
 (g,\tilde\psi_{j,L})\!=\!(\Cal L_{\Xi,\Omega_M}(\Cal L_{\Xi,\Omega_M}+\eb\Pi_{\Xi,L})^{-1}g,\tilde\psi_{j,L}),\, \forall j\bigg\}.
\end{equation}
In fact, if $g$ satisfies~\eqref{4.I}, then $(u_g,\tilde\psi_{j,L})=0$
for all $j$, therefore
$\Cal L_{\Xi,\Omega_M}u_g=(\Cal L_{\Xi,\Omega_M}+\varepsilon\Pi_{\Xi,L})u_g=g$,
and hence $g\in I$.  Conversely, let $g\in I$. Then, by the definition of $I$,  there exists $v$
such that $(v,\tilde\psi_{j,L})=0$ and
$\Cal L_{\Xi,\Omega_M}v=(\Cal L_{\Xi,\Omega_M}+\eb\Pi_{\Xi,L})v=g$.
Hence, $v=(\Cal L_{\Xi,\Omega_M}+\eb\Pi_{\Xi,L})^{-1}g=u_g$
and~\eqref{4.I} holds.

Thus, the operator $\Cal L_{\Xi,\Omega_M}$ or, which is the same as
$\Cal L_{\Xi,\Omega_M}(1-\Pi_{\Xi,L})$,
 is invertible as the operator from $H_{hyp}\cap D(\Cal L_{\Xi,\Omega_M})\to I$,
  and the inverse is given by $g\to u_g$.

To finish the proof of the lemma, we need to verify that the operator
\begin{equation}\label{4.II}
(1-\Pi_{\Xi,L}): I\to H_{hyp}
\end{equation}
is also uniformly invertible. Let us solve the equation $(1-\Pi_{\Xi,L})\hat g=g$,
 where $g\in H_{hyp}$ and $\hat g\in I$. Then, since
 $L^2_\sigma=H_{hyp}\oplus H_{neu}$, the solution has the form
$$
\hat g=g+\sum_j A_j\varphi_{j,L}\in I
$$
for some numbers $A_j$ (in a bounded domain there are only
 finitely many of them, but for $\Omega=\R^2$ this number may
 be infinite, then we need to take $A\in l^2$ or $A\in l^\infty$).
 Using \eqref{4.I}, we write down the system of equations for~$A_j$:
\begin{equation}\label{4.A}
A_k-\alpha_L\sum_jA_j(\Cal L_{\Xi,\Omega_M}u_{\varphi_{j,L}},\tilde \psi_{k,L})=\alpha_L(\Cal L_{\Xi,\Omega_M}u_g,\tilde \psi_{k,L}),\ \forall k,
\end{equation}
where
$$
 u_{\varphi_{j,L}}:=(\Cal L_{\Xi,\Omega_M}+\eb\Pi_{\Xi,L})^{-1}\varphi_{j,L},
 $$
We need to show that the left-hand side is uniformly close to
the identity operator in $l^2$. To this end, we need to estimate
the $l^2$-norm of the corresponding matrix operator
$\mathcal B$, $\{b_j\}_{\xi_j\in\Xi}=\{\sum_kB_{jk}a_k\}_{\xi_j\in\Xi}$. We set
$$
B:=B_1+B_2,\quad B_1:=\sup_{j}\sum_{k}B_{jk},\ B_2:=\sup_k\sum_jB_{j,k},
$$
where $B_{jk}:=|(\Cal L_{\Xi,\Omega_M}u_{\varphi_{j,L}},\tilde\psi_{k,L})|$,
$\xi_j,\xi_k\in\Xi$. Then we clearly have
$$
\|\mathcal B\|_{l^1\to l^1}\le B_1,\quad \|\mathcal B\|_{l^\infty\to l^\infty}\le B_2,
$$
so that H\"older's inequality (or complex interpolation, see e.g., \cite{Trib})
gives that
$$
\|\mathcal B\|_{l^2\to l^2}\le\sqrt{B_1B_2}.
$$

We first consider the terms $B_{jk}$ with $k\ne j$. Integrating by parts and using that $P_{\Omega_M}u_{\varphi_{j,L}}=u_{\varphi_{j,L}}$, we have
\begin{multline}\label{4.Bjk}
B_{j,k}=(u_{\varphi_{j,L}}, \hat L_{\Xi}^*(P_{\Omega_M}\tilde\psi_{k,L}))+\\
+(\partial_nu_{\varphi_{j,L}},P_{\Omega_M}\tilde\psi_{k,L})_{\partial\Omega_M}-\lambda(u_{\varphi_{j,L}},\tilde \psi_{k,L}),
\end{multline}
where
$$
\hat L_{\Xi}^*v:=\Delta v+2(\bar u_\Xi,\Nx)v+\bar u_\Xi^\bot\Rot v.
$$
 To estimate the quantity $B$, we need one more proposition.

\begin{proposition}\label{Prop4.weight} Let $L$ and $N$ be large enough. Then $u_{\varphi_{j,L}}=\tilde u_{\varphi_{j,L}}+\hat u_{\varphi_{j,L}}$ where
\begin{equation}\label{4.u}
\|\tilde u_{\varphi_{j,L}}\|_{H^2(B^1_{x_0})}\le C\theta(\xi_j-x_0)^{5/6},\ \
\|\hat u_{\varphi_{j,L}}\|{\color{black}_{H^2(\Omega_M)}}\le CM^{-2}
\end{equation}
and the constant $C$ is independent of $x_0$, $j$, $L$ and $M$.
\end{proposition}
\begin{proof}[Proof of the proposition] We set
$\tilde u_{\varphi_{j,L}}:=(\Cal L_\Xi+\eb\Pi_{\Xi,L})^{-1}\varphi_{j,L}$.
Then, due to Theorem \ref{Th2.main} (and estimate in the $L^2A^\infty$-spaces),
this function satisfies estimate \eqref{4.u}. The remainder $\hat u_{\varphi_{j,L}}$
obviously solves
\begin{equation}\label{4.auu}
\Cal L_{\Xi,\Omega_M}\hat u_{\varphi_{j,L}}+\eb\Pi_{\Xi,L}\hat u_{\varphi_{j,L}}=0,\ \ \hat u_{\varphi_{j,L}}\big|_{\partial\Omega_M}=\tilde u_{\varphi_{j,L}}\big|_{\partial\Omega_M}.
\end{equation}
{\color{black} Here we implicitly used that
$$
(\Cal L_{\Xi,\Omega_M}-\Cal L_{\Xi})\tilde u=
(P_{\Omega_M}-P)((\Delta-\lambda)\tilde u-
\divv(\bar u_\Xi\otimes\tilde u+\tilde u\otimes\bar u_\Xi))=\nabla Q
$$
and, therefore, $P_{\Omega_M}(\Cal L_{\Xi,\Omega_M}-\Cal L_\Xi)=0$.}
To proceed further, we need to prove that
\begin{equation}\label{4.nice}
\|\tilde u_{\varphi_{j,L}}\|_{H^{3/2}(\partial\Omega_M)}\le CM^{-5/2}.
\end{equation}
Indeed, if this estimate is obtained,  we can construct the extension from the boundary, which satisfies
$$
\|\tilde u_{ext}\|_{H^2}\le CM^{-2}.
$$
After that, we use Theorem \ref{Th3.main} and get the desired estimate for $\hat u_{\varphi_{j,L}}$. Thus, we only need to verify \eqref{4.nice}. To this end, we write
{\color{black}
\begin{multline}\label{4.wrong}
\tilde u_{\varphi_{j,L}}\!=\!P\divv(\Dx-\lambda)^{-1}(\bar u_\Xi\otimes \tilde u_{\varphi_{j,L}}+\tilde u_{\varphi_{j,L}}\otimes\bar u_\Xi)-\\-(\Dx-\lambda)^{-1}\Pi_{\Xi,L}\tilde u_{\varphi_{j,L}}+(\Dx-\lambda)^{-1}\varphi_{j,L}:=I_1+I_2+I_3.
\end{multline}
}
We note that, according to the construction of $\Xi$,
there exist $0<\kappa<\kappa'$ such that $|\xi_i|<\kappa M$
 for all $i$ and $|x|>\kappa'M$ for all $x\in\Cal O_1(\partial\Omega_M)$.
 Thus, for sufficiently large $N$, we have $(\Pi_{\Xi,L}\tilde u_{\varphi_{j,L}})(x)=0$
 for $|x|\ge \kappa''M$ with $\kappa<\kappa''<\kappa'$. Let
$$
w(x):=\max\{1,e^{\alpha(|x|-\kappa''M)}\}.
$$
Then, $w$ is a weight function of exponential growth rate $\alpha$ and,
 for $\alpha>0$ small enough, we have the $L^1$-weighted regularity
 estimate for the operator $(\Dx-\lambda)$, namely
\begin{multline*}
\|I_2\|_{L^1_w}\le C\|\Pi_{\Xi,L}\tilde u_{\varphi_{j,L}}\|_{L^1}
\\\le C_1\int_{x,y}
\sum_k\theta^{5/6}(\xi_j-y)\theta(\xi_k-y)^{5/6}\theta(\xi_k-x)^{5/6}\,dx\,dy
\\\le C_2\int_{x,y}\theta(\xi_j-y)^{5/6}\theta(x-y)^{5/6}\,dx\,dy\le C_3,
\end{multline*}
where the constants $C_i$ are independent of $j$ and $M$, see e.g., \cite{EfZ}.
Here we have also implicitly used the proved estimate \eqref{4.u} for
$\tilde u_{\varphi_{j,L}}$ Therefore, {\color{black} differentiating \eqref{4.wrong} several times and using the Sobolev embedding theorem, we get}
$$
|I_2(x)|\le Ce^{-\alpha(\kappa'-\kappa'')M},\ \ x\in\Cal O_1(\partial\Omega_M)
$$
and the same is true for all derivatives of $I_2$. This shows that $I_2(x)$ is exponentially small near the boundary with all its derivatives and, in particular, \eqref{4.nice} holds for this part of $\tilde u_{\varphi_{j,L}}$. {\color{black} The term $I_3$ can be estimated analogously, so we only need to estimate $I_1$.
\par
We write $I_1:=P\div I_4$.} Arguing as before, we see that
$$
\|I_{\color{black}4}\|_{L^1_w}\le C\|\bar u_\Xi\otimes u_{\varphi_{j,L}}+u_{\varphi_{j,L}}\otimes\bar u_\Xi\|_{L^1}\le C_1
$$
and the same is true for the derivatives. We recall that $P\div=N_1+N_2$,
where $N_1$ is a ``local" operator which is bounded from $H^{1}_w$ to $L^2_w$,
 so the term $N_1 I_{\color{black}4}$ is exponentially small near the
 boundary together with all its derivatives and we only need to consider
 the non-local part $N_2I_{\color{black}4}$, where the kernel of $N_2$ is
 smooth and satisfies $|K_{N_2}(z)|\le C\theta(z)$. Let $\delta\in(\kappa'',\kappa')$. Then
\begin{multline*}
|N_2I_{\color{black}4}(x)|\le C\int_y I_{\color{black}4}(y)\theta(x-y)\,dy=\int_{|y|<\delta M}I_{\color{black}4}(y)\theta(x-y)\,dy+\\+\int_{|y|>\delta M} I_{\color{black}4}(y)\theta(x-y)\,dy:=J_1+J_2.
\end{multline*}
Since ${\color{black}I_4}$ is uniformly bounded in $L^1_w$, the term $J_2$ is exponentially small with respect to $M$, so we only need to look at $J_1$. If $x\in\Cal O_1(\partial\Omega_M)$, we have
$$
|J_1(x)|\le C(\kappa'-\delta)^{-3}M^{-3}\|I_{\color{black}4}\|_{L^1}\le C_2M^{-3}
$$
and the same estimate holds for its derivatives.
Thus, $J_1(x)$ is  of order $M^{-3}$ near the boundary.
 Thus, estimate \eqref{4.nice} is proved and the proposition is also proved.
\end{proof}

We now return to the proof of the lemma. We start with the
boundary term in \eqref{4.Bjk}. Indeed, due to \eqref{4.nice} and \eqref{4.u},
{\color{black}and the fact that in view of \eqref{4.gradP}
\begin{multline}
\|P_{\Omega_M}\tilde\psi_{j,L}\|_{H^1(\Omega_M)}^2=\|P_{\Omega_M}\psi_{j,L}\|^2_{L^2}+\|\Nx P_{\Omega_M}\psi_{j,L}\|^2_{L^2}\le\\\le \|\psi_{j,L}\|^2_{L^2}+\|\Nx \psi_{j,L}\|^2_{L^2}+CM^{-2}\le C
\end{multline}
we obtain, using also \eqref{4.u} and \eqref{4.nice}, that
\begin{multline}
|(\partial_nu_{\varphi_{j,L}},P_{\Omega_M}\tilde\psi_{j,L})_{\partial\Omega_M}|\le\\\le C(\|\hat u_{\varphi_{j,L}}\|_{H^2(\Omega_M)}+\|\tilde u_{\varphi_{j,L}}\|_{H^1(\partial\Omega_M)})\|P_{\Omega_M}\tilde\psi_{j,L}\|_{H^1(\Omega_M)}\le CM^{-2},
\end{multline}}
and the impact of this term to the quantity $B$ is of order $N^2M^{-2}=L^{-2}$.
\par
Let us estimate the other two terms in \eqref{4.Bjk}. In the case, $j\ne k$,
we replace $u_{\varphi_{j,L}}$ and $P_{\Omega_M}\tilde\psi_{k,L}$ by
$\tilde u_{\varphi_{j,L}}$ and $P\tilde \psi_{k,L}$, respectively.
Then, using Propositions \ref{Prop4.M} and \ref{Prop4.weight},
we see that the error of such approximation will be of order $M^{-2}$
and the impact of this error to the quantity $B$ is of order $L^{-2}$.
Thus, we need to estimate the terms
$(\tilde u_{\varphi_{j,L}},\hat L_\Xi^*(P\tilde\psi_{k,L}))$.
Using \eqref{4.u} and the fact that the kernel of $\Nx\circ P$
decays cubically, we arrive at
$$
|((\tilde u_{\varphi_{j,L}},\hat L_\Xi^*(P\tilde\psi_{k,L}))|\le
 C(\theta(x-\xi_j)^{5/6},\theta(x-\xi_k)^{5/6})\le C_1\theta(\xi_j-\xi_k)^{5/6}
$$
and taking the sum with respect to $j\ne k$ (or $k\ne j$),
we see that the impact of these terms to the quantity $B$ is of order $L^{-5/2}$.
The analogous estimate for the third term in \eqref{4.Bjk} is immediate, so we
see that the impact of $B_{jk}$ with $j\ne k$ is of order $L^{-2}$. Therefore,
it only remains to consider the case $j=k$. Using that
$\bar u_\Xi=\sum_i \bar u_i$, we write
\begin{multline}
B_{jj}=(\tilde u_{\varphi_{j,L}},\Cal L_j^*(P\tilde \psi_{j,L}))+
 O(M^{-2})+\\+\sum_{i\ne j}\((\tilde u_{\varphi_{j,L}},2(\bar u_i,\Nx)
 (P\tilde\psi_{j,L})) +
 (\tilde u_{\varphi_{j,L}},\bar u_i^\bot\Rot(P\tilde\psi_{j,L}))\).
\end{multline}
Using the first estimate in \eqref{4.u} and that the support of $\bar u_i$ is
localized near $\xi_i$, we
see that the second term on the right is bounded by
$$
 C\sum_{i\ne j}(\theta(x-\xi_i)^{5/6},\theta(x-\xi_j)^{5/6})\le C_1\sum_{i\ne j}\theta(\xi_i-\xi_j)^{5/6}\le C_2 L^{-5/2}.
$$
Finally, using estimate \eqref{1.app} for the first term, we see that
the impact of terms $B_{jj}$ to $B$ is of order $L^{-1/2}$. Combining
together the obtained estimates, we see that
$$
B\le CL^{-1/2}.
$$
Since $B$ majorates the norm of the matrix operator $(B_{jk})_{\xi_j,\xi_k\in\Xi}$ in $l^2$, we see that the left-hand side of \eqref{4.A}, we see that the operator in the left-hand side is indeed is uniformly close to identity and, by this reason, is invertible if $L$ is large enough.
  We also mention that arguing analogously, see \eqref{est} and below, we get
$$
\sum_{k}(\Cal L_{\Xi,\Omega_M}u_{g},\tilde\psi_{k,L})^2\le CL^{-1}\|u_g\|^2_{H^2}\le C_1L^{-1}\|g\|_{L^2}^2
$$
and, therefore, \eqref{4.A} is uniquely solvable and
$$
\sum_k A_k^2\le C\|g\|^2_{L^2}.
$$
This gives us the uniform invertibility of $(1-\Pi_{\Xi,L}): I\to H_{hyp}$
 and   proves the lemma for the case of $\Omega=\Omega_M$.
 Note that in the proof we essentially used the fact that $\#\Xi\sim c_1N^2$.
 However, this was used to replace the projector $P_{\Omega_M}$ by $P$ only.
 Since in the case $\Omega=\R^2$, we do not have this projector at all,
 the proof remains valid for any number of vortices.
\end{proof}

We are now ready to finish the proof of the theorem.
To this end, we write down the equations for the operators
 $K_{neu}$ and $K_{hyp}$ using the invariance condition.
 Namely, $K_{neu}$ should satisfy
\begin{equation*}
\(\begin{matrix} \Cal L_{11}&\Cal L_{1,2}\\\Cal L_{2,1}&\Cal L_{22}\end{matrix}\)\(\begin{matrix}K_{neu}\\1\end{matrix}\)=\(\begin{matrix}\Cal L_{11}K_{neu}+\Cal L_{12}\\\Cal L_{21}K_{neu}+\Cal L_{22}\end{matrix}\)=\(\begin{matrix}K_{neu}(\Cal L_{21}K_{neu}+\Cal L_{22})\\\Cal L_{21}K_{neu}+\Cal L_{22}\end{matrix}\),
\end{equation*}
which gives us the desired Riccatti type operator equation for $K_{neu}$:
\begin{equation}
K_{neu}=\Cal L_{11}^{-1}K_{neu}\Cal L_{21}K_{neu}+\Cal L_{11}^{-1}K_{neu}\Cal L_{22}
{\color{black}-\Cal L_{11}^{-1}\Cal L_{12}}.
\end{equation}
Using {\color{black} Corollary \ref{Cor4.small} and Lemma~\ref{Lem4.inv}}, we see that the right-hand side of this equation is uniformly small if $K_{neu}$ is uniformly bounded, so we can find the uniformly bounded solution of this equation, say in $\Cal L(L^2,H^2)$ via the Banach contraction theorem. Moreover, the restriction of the operator $\Cal L_\Xi$ (or $\Cal L_{\Xi,\Omega_M}$) to the base $H_{hyp}$ is given by $\Cal L_{21}K_{neu}+\Cal L_{22}$ and we see that it is of order $O(L^{-1/2})$. This gives the second spectral condition \eqref{4.spec}.
\par
Let us construct now the operator $K_{hyp}$. Arguing analogously,
we derive the Riccatti equation for it
$$
K_{hyp}=\Cal L_{22}K_{hyp}\Cal L_{11}^{-1}-
K_{hyp}\Cal L_{12}K_{hyp}\Cal L_{11}^{-1}+\Cal L_{21}\Cal L_{11}^{-1},
$$
and the Banach contraction theorem gives us the solution of this
equation in $\Cal L(L^2,H^2)$. Also, the projection of $\Cal L_\Xi$
(or $\Cal L_{\Xi,\Omega_M}$) to the base $H_{hyp}$ is given by
$\Cal L_{11}+\Cal L_{12}K_{hyp}$, which is a small perturbation
of an invertible operator. This gives us the first condition of
\eqref{4.spec} and finishes the proof of the theorem.
\end{proof}
\section{Main results}\label{s5}
In this section, we derive  lower bounds for the attractors
dimension in two cases: \begin{enumerate}
                          \item Classical Navier--Stokes equation in the
2D bounded smooth simply connected domain $\Omega$;
                          \item Navier--Stokes problem with linear Ekman damping in the whole plane.
                        \end{enumerate}

 We start with the
first case:
\begin{equation}\label{5.NS}
\Dt u+\divv(u\otimes u)+\Nx p=\nu\Dx u+f,\ \
\divv u=0,\ u\big|_{t=0}=u_0,\ u\big|_{\partial\Omega}=0,
\end{equation}
where $\nu>0$ is a small parameter and $f\in L^2$ is a given external force.
It is well-known that this equation generates a dissipative semigroup in
the space $\Phi=L^2_\sigma(\Omega)$ via
$$
S(t)u_0:=u(t),\ \  S(t):\Phi\to\Phi.
$$
Moreover, this semigroup possesses a global attractor $\Cal A$ of
 finite fractal dimension, see \cite{BV,Tem}, For the convenience of the reader,
 we recall that the set $\Cal A$ is a global attractor for the semigroup
 $S(t)$ in $\Phi$ if
\par
a) $\Cal A$ is a compact and strictly invariant (i.e. $S(t)\Cal A=\Cal A$)
subset of the phase space $\Phi$;
\par
b) $\Cal A$ attracts the images of all bounded sets of $\Phi$ as time tends
to infinity, i.e. for every bounded set $B\subset\Phi$ and every
neighbourhood $\Cal O(\Cal A)$ of $\Cal A$ there exists time
$T=T(\Cal B,\Cal O)$ such that
$$
S(t)B\subset\Cal O(\Cal A),\ \ \forall t\ge T.
$$
We also recall that by the Hausdorff criterion, for any $\eb>0$, the compact
set $\Cal A$ can be covered by finitely many $\eb$-balls. Let
$N_\varepsilon(\Cal A,\Phi)$
be the minimal number of such balls, then the fractal (box-counting)
dimension of $\Cal A$ is defined as follows:
$$
\dim_f(\Cal A)=\dim_f(\Cal A,\Phi):=
\limsup_{\eb\to0}\frac{\log_2 N_\varepsilon(\Cal A,\Phi)}{\log_2\frac1\eb},
$$
see \cite{Tem,BV,Rob} for more details. As mentioned in the introduction,
the bounds for the fractal dimension of the attractor $\Cal A$ are
usually expressed in terms of the dimensionless quantity, called the
Grashof number $$G:=\frac{|\Omega|\|f\|_{L^2(\Omega)}}{\nu^2}$$ and
the best known upper bound is (see e.g. \cite{Tem})
\begin{equation}\label{5.bad}
\dim_f(\Cal A,\Phi)\le \color{blue}\tilde cG_{-1}\le c G\color{black},
\end{equation}
see also \cite{IKZ} for the explicit values of the absolute constant $c$.
Moreover, in the case of periodic boundary conditions this estimate can be
essentially improved \cite{CFT, Tem}:
\begin{equation}\label{5.good}
\dim_f(\Cal A,\Phi)\le c_{per}G^{2/3}\ln(1+G)^{1/3},
\end{equation}
see \cite{IZ25} for the bound for the constant $c_{per}$. The following theorem,
which gives the lower bounds of this quantity in the case of a bounded domain
is one of the main results of the paper.
\begin{theorem}\label{Th5.main1} Let $\Omega$ be a smooth simply connected
bounded domain in $\R^2$. Then, for every sufficiently small $\nu>0$,
there exists $f=f_\nu\in L^2_\sigma(\Omega)$ such that
\begin{equation}\label{5.lob}
\dim_f(\Cal A,\Phi)\ge \bar c G^{2/3},
\end{equation}
for some absolute constant $\bar c$. Moreover, $\|f_\nu\|_{L^2}\sim\nu^{1/2}$
and $G\sim\nu^{-3/2}$.
\end{theorem}
\begin{proof} First, we do the standard scaling

$$
{\color{black}
x=\nu^{1/2}x',\quad \tilde u(t,x')=\nu^{-1/2}u(t,x),\quad
\tilde p(t,x')=\nu^{-1}p(t,x),
}
$$
which transforms our equation~\eqref{5.NS} to the equation
for $\tilde u$ and $\tilde p$ in a big domain $\Omega_M=M\Omega$,
$M=\nu^{-1/2}$ but with new $\nu=1$
and $\tilde f(x')=\nu^{-1/2}f(x)$. After that, we take a
Vishik vortex $\bar u_0(x)$ which satisfies the assumptions
of section \ref{s1}, fix a sufficiently large $L$ and take $N=M/L$
in such a way that the assumptions of Theorem \ref{Th4.main} are
satisfied for the multi-vortex $\bar u_\Xi$ with $\#\Xi\sim N^2$.
Finally, we fix the external force $f_\Xi(x')$ produced by the
multi-vortex~$\bar u_\Xi(x')$.
\par
By the construction, the multi-vortex $\bar u_\xi$ will be an
equilibrium for the scaled Navier--Stokes equation \eqref{5.NS}
with $\nu=1$ and, due to Theorem \ref{Th4.main}, the instability
index of this equilibrium is not less than $N^2$.  Thus, since the
attractor always contains the unstable manifolds of any
equilibrium, the dimension of the attractor is not less than $N^2$.
Since the Grashof number is invariant with respect to this scaling,
we have
$$
\dim_f(\Cal A,\Phi)\ge N^2,\ \ |\Omega_M|\sim M^2\sim N^2,\ \
\|f_\Xi\|_{L^2(\Omega_M)}\sim N,\ \ G\sim N^{3}
$$
and, therefore, $\dim_f(\Cal A,\Phi)\ge \bar c G^{2/3}$. {\color{black} Indeed, since the supports of $f_i$ are localized,
$$
\|f_\Xi\|_{L^2}^2=\sum_{\xi_i\in\Xi}\|f_i\|^2_{L^2}=\#\Xi\|f_0\|^2_{L^2}\sim N^2.
$$
In a general case, $f_i$ are not localized, we have $f_i=P_{\Omega}\hat f_i$, where $\hat f_i$ are localized. Since the norm of $P_{\Omega_M}$ in $L^2$ is one, we have the same estimate.}
\par
Returning to the original domain $\Omega$ we have
\begin{equation}\label{oscillate}
f_\nu(x)=\nu^{1/2}f_{\Xi(\nu)}\left(\frac x{\nu^{1/2}}\right),
\end{equation}
which completes the proof of the theorem.
\end{proof}
 We now discuss the lower bounds through the quantity $G_{-1}$.
\begin{corollary}\label{Cor7.m1} Let us consider exactly the same construction of the function $f_\nu$ as in the previous theorem. Then
$$
\dim_f(\Omega)\ge \tilde cG_{-1}
$$
for some absolute constant $\tilde c>0$.
\end{corollary}
\begin{proof} Indeed, by construction, we have $f_{\Xi}=\sum_i f_i$, where $f_i=P_{\Omega}\tilde f_i$ and $\tilde f_i=\divv F_i$ with fully localized tensors $F_i$. We then write
$$
f_\Xi=\Nx p+\divv(\sum_i F_i)=\Nx p+\tilde f_\Xi
$$
with fully localized $F_i$. Thus,
$$
\|\tilde f_\Xi\|_{\dot H^{-1}(\Omega)}\le \|\sum_i F_i\|_{L^2}\sim N^{1/2}.
$$
This, in turn, gives $G_{-1}\sim N\sim \nu^{-1}$ and finishes the proof of the corollary.
\end{proof}
\begin{remark}\label{Rem7.GG}
The lower bound \eqref{5.lob} was known before for
the case of periodic boundary conditions,
see \cite{Liu}. It also holds for the sphere $\mathbb S^2$ as
the numerical analysis of the corresponding spectral problem suggests,
see \cite{I-Kolm2004}. Thus, together with \eqref{5.good}
(which holds in both cases),
 this estimate gives almost
sharp upper and lower bounds for the attractor's dimension  in terms of the Grashov number $G$.  Recall
that the upper bounds utilize good estimates for the vorticity
equation and the lower bounds use the Kolmogorov flows, non of
which works for the case of bounded domains and Dirichlet boundary
conditions. By this reason,   upper
bounds \eqref{5.bad} in terms of $G$ are essentially worse (since we are unable to utilize the vorticity equation as well as $L^2$-regularity of $f$). Also, no lower bounds at all were known before.
\par
As we see from Theorem \ref{Th5.main1} and Corollary \ref{Cor7.m1}, our approach gives
the lower bounds in terms of $G$ for Dirichlet boundary conditions of the
same type as for the periodic ones, but the upper and lower bounds in these case are essentially different and the estimate is not sharp
\par
In contrast to this, if we use a bit more natural dimensionless quantity $G_{-1}$ the upper and lower bounds become sharp and this resolves the long-standing open problem of finding sharp bounds for the attractor dimension for the case of Dirichlet boundary conditions. The surprising improvement given by passing from $G$ to $G_{-1}$ can be explained as follows: as we already mentioned in the introduction, the upper bound \eqref{5.bad} is actually obtained for $G_{-1}$ in the right-hand side and then $G_{-1}$ is replaced by $G$ using the Poincare inequality
\begin{equation}\label{7.Puan}
\|f\|_{\dot H^{-1}}\le C_{\Omega}\|f\|_{L^2},\quad C_\Omega\le
\left(\frac{|\Omega|}{2\pi}\right)^{1/2}.
\end{equation}
However, our lower bounds utilize highly oscillating functions of the form \eqref{oscillate} and for such functions we have much better relation
$$
\|f_\nu\|_{\dot H^{-1}}\sim C_{\Omega}\nu^{1/2}\|f_\nu\|_{L^2}
$$
and, by this reason, the passage from $G_{-1}$ to $G$ via the Poincare inequality may loose sharpness. However, it is still unclear whether or not we may find an alternative way to write better (than via Poincare) estimate for the dimension of the attractor in terms of $G$ for Dirichlet BC (e.g., as it done for the case of periodic BC).
\par
Note that our method gives an alternative to the Kolmogorov flows
for periodic boundary conditions as well. Indeed, starting from
the Vishik vortex $\bar u_0$, consider the periodic set $\Xi$,
which will give us a periodic profile $\bar u_\Xi$, which, in turn,
will be the equilibrium replacing the Kolmogorov flows
(there is a minor technical thing that we need
$\bar u_\Xi$ to have zero mean, but this is easily
solvable considering for instance a vortex-antivortex pair).
After that, we consider only space periodic perturbations from
the space $L^2_{\sigma,per}(\R^2)\subset L^2_{b,\sigma}(\R^2)$
and use Theorem \ref{Th4.main}. This approach is probably more
technical, especially if we take into account the construction
of the Vishik vortex, but looks more transparent than the
Kolmogorov flows at least on the level of the ideas.
\end{remark}
We now turn to our second main result, namely, the
Navier--Stokes system with linear Ekman damping  on the whole plane $\Omega=\R^2$:
\begin{equation}\label{5.ENS}
\Dt u+\divv(u\otimes u)+\Nx p+\mu u=\nu\Dx u+f,\ \ \divv u=0,\ \ u\big|_{t=0}=u_0,
\end{equation}
where $\nu,\mu>0$ are small parameters and $f\in H^1(\R^2)$.
It is known, see \cite{IPZ}, that this problem generates a
dissipative semigroup in the phase space $\Phi:=L^2_\sigma(\R^2)$,
which possesses a global attractor $\Cal A$ of finite fractal dimension.
Moreover, the analogue of the Grashof number here is the following
dimensionless quantity
$$
G_1:=\frac{\|\Rot f\|_{L^2}^2}{\mu^3\nu}
$$
and
the upper bounds obtained in \cite{IPZ} read
\begin{equation}\label{5.usharp}
\dim(\Cal A,\Phi)\le c_1 G_1,
\end{equation}
for some absolute constant $c_1$. Our second main result
gives the sharp lower bound for the above upper bound.
 \begin{theorem}\label{Th5.main2} For any $\mu,\nu>0$
 and any sufficiently large number $r$, there exists an
  external force $f_{\mu,\nu,r}\in H^1_\sigma(\R^2)$ such that
 \begin{equation}\label{5.lsharp}
 \dim_f(\Cal A,\Phi)\ge c_2 G_1
 \end{equation}
 for some absolute constant $c_2$. Moreover, $c_3r\le G_1\le c_4r$ for some absolute constants $c_3$ and $c_4$.
 \end{theorem}
 \begin{proof}{\color{black} As in the proof of Theorem \ref{Th5.main1},
 we use scaling
$$
t=\frac {t'}\mu,\quad
x=\left(\frac\nu\mu\right)^{1/2}x',\quad \tilde u(t',x')=
\frac1{(\mu\nu)^{1/2}}u(t,x),\quad
\tilde p(t',x')=\nu^{-1}p(t,x),
$$
which reduces the problem to the case $\mu=\nu=1$
and with forcing term $\tilde f(x')=\frac1{\mu^{3/2}\nu^{1/2}}f(x)$.
 Let again $\bar u_0$ be the Vishik vortex satisfying the
 assumptions of section \ref{s1} and using scaling, we may
 assume without loss of generality that the unstable eigenvalue
 $\lambda$ satisfies $\Ree\lambda>1$.

 After that, we fix $f'_0:=f_0+\bar u_0$ in order to
 make $\bar u_0$ an equilibrium for the scaled equations
 \eqref{5.ENS} and  fix a
 sufficiently large $L$ for a set  of vertices~ $\Xi$
 with $\#\Xi=N$ in such a way that
 Theorem \ref{Th4.main} holds.
Estimating the term $\|\Rot f'_{\Xi}\|_{L^2}$ as before using that the supports of $\Rot f_i$ do not intersect, we have
$$
\dim_f(\mathcal A,\Phi)\ge  N,\ \ G_1=\|\Rot f'_\Xi\|^2_{L^2}=\#\Xi\|\Rot f'_{\Xi}\|^2_{L^2}\sim N
$$
 Returning to the original variables we have
$$
f_{\nu,\mu}(x)=\mu^{3/2}\nu^{1/2}f'_{\Xi}\left(\frac{\mu^{1/2} x}{\nu^{1/2}}\right),
$$
which completes the proof of the theorem.}
 \end{proof}
\begin{remark} In the case of periodic boundary conditions,
the sharp estimates \eqref{5.usharp} and \eqref{5.lsharp} are
obtained in \cite{IMT} under the extra condition that $\nu\ll\mu$
for the lower bounds. These lower bounds for the torus can be also
reproduced by our method and the extra restriction is related with
the fact that after scaling we will have the equation in the domain
 $\Omega_M=M\Omega$ with $M=\mu^{1/2}\nu^{-1/2}$ and it is crucial
for us that the scaled domain must be large. Of course, this condition
is not necessary when $\Omega=\R^2$. We also note that our method
also gives estimate \eqref{5.lsharp} for the case of equation
\eqref{5.ENS} in a bounded domain with {\color{black}stress-free} boundary
conditions. We return to this case somewhere else.
\end{remark}

{\color{black}{\bf Acknowledgement.} This work was done with the
financial support from the Russian Science Foundation (grant no.
23-71-30008 (AK), grant no. 25-11-20060 (SZ), and
 grant
no. 25-71-30001 (AI)).}

\end{document}